\documentclass[11pt,a4paper,reqno]{amsart}

\usepackage{times}
\usepackage{xspace}
\usepackage[inline]{enumitem}
\usepackage{enumerate}
\usepackage[numbers,sort&compress,longnamesfirst]{natbib}
\usepackage[hmargin=2.0cm,vmargin=2.8cm]{geometry}
\usepackage[colorlinks = true, linkcolor = blue, citecolor = blue]{hyperref}
\usepackage{amsfonts, comment}
\usepackage{amssymb, amsmath, latexsym, mathtools}
\usepackage{amsthm}
\usepackage[capitalize]{cleveref}
\usepackage{color, graphicx} 
\usepackage{todonotes}
\usepackage{lipsum}  
\usepackage{bbm}

\newlist{myitems}{enumerate}{1}
\setlist[myitems]{label=\arabic*, font=\bfseries, resume}
\numberwithin{equation}{section}
\theoremstyle{plain}
\newtheorem{theorem}{Theorem}[section]
\newtheorem{lemma}[theorem]{Lemma}
\newtheorem{proposition}[theorem]{Proposition}

\theoremstyle{definition}
\newtheorem{definition}[theorem]{Definition}
\newtheorem{remark}[theorem]{Remark}

\newtheorem{example}[theorem]{Example}
\newtheorem{assumption}{Assumption}

\crefname{assumption}{Assumption}{Assumptions}
\Crefname{assumption}{Assumption}{Assumptions}
\newenvironment{namedassumption}[2][]
  {\renewcommand\theassumption{#2}\begin{assumption}[#1]}
  {\end{assumption}}
  
\def\<{\langle}
\def\>{\rangle}
\def\e{{\varepsilon}}
\def\ud{\mathrm d}
\def\ue{\mathrm e}

\DeclareMathOperator*{\argmin}{arg\,min}

\newcommand{\ip}[1]{\left\langle#1\right\rangle}
\newcommand{\p}{\mathbb{P}}
\newcommand{\E}{\mathbb{E}}
\newcommand{\T}{\mathcal{T}}

\newcommand{\cL}{\mathcal{L}}

\newcommand{\cH}{\mathcal{H}}
\newcommand{\R}{\mathbb{R}}

\newcommand{\cV}{\mathcal{V}}

\newcommand{\D}{\mathcal{D}}

\newcommand{\norm}[1]{\left\|#1 \right\|} 

\newcommand{\abs}[1]{\left|#1\right|} 

\begin{document}

\title[Convergence of Deep Gradient Flow Methods for PDE]{Convergence of the generalization error for Deep Gradient Flow Methods for PDE\MakeLowercase{s}}

\author[C. Liu]{Chenguang Liu}
\author[A. Papapantoleon]{Antonis Papapantoleon}
\author[J. Rou]{Jasper Rou}

\address{Delft Institute of Applied Mathematics, EEMCS, TU Delft, 2628 Delft, The Netherlands.}
\email{C.Liu-13@tudelft.nl}

\address{Delft Institute of Applied Mathematics, EEMCS, TU Delft, 2628 Delft, The Netherlands \& Institute of Applied and Computational Mathematics, FORTH, 70013 Heraklion, Greece}
\email{a.papapantoleon@tudelft.nl}

\address{Delft Institute of Applied Mathematics, EEMCS, TU Delft, 2628 Delft, The Netherlands}
\email{J.G.Rou@tudelft.nl}



\begin{abstract}
The aim of this article is to provide a firm mathematical foundation for the application of deep gradient flow methods (DGFMs) for the solution of (high-dimensional) partial differential equations (PDEs).
We decompose the generalization error of DGFMs into an approximation and a training error. 
We first show that the solution of PDEs that satisfy reasonable and verifiable assumptions can be approximated by neural networks, thus the approximation error tends to zero as the number of neurons tends to infinity.
Then, we derive the gradient flow that the training process follows in the ``wide network limit'' and analyze the limit of this flow as the training time tends to infinity. These results combined show that the generalization error of DGFMs tends to zero as the number of neurons and the training time tend to infinity.
\end{abstract}

\maketitle


\section{Introduction}

Deep learning methods for the solution of high-dimensional partial differential equations (PDEs) have gained tremendous popularity in the last few years, since they can tackle equations in dimensions that were not attainable by classical methods, such as finite difference and finite element schemes.
This ability allows the modeling of more realistic phenomena across various fields of science and technology, including engineering, biology, economics, and finance.
The seminal articles of \citet{sirignano2018dgm} on the Deep Galerkin Method (DGM) and of \citet{Raissi2018} and \citet{PINN} on physics-informed neural networks (PINNs), building on the earlier work of \citet{Lagaris1998} and \citet{Lagaris2000}, incorporate the PDE residual and the initial and boundary conditions into the loss function of a neural network.
This loss function is then minimized by stochastic gradient descent (SGD) or one of its variants, \textit{e.g.} AdaGrad, Adam or RMSProp; see, for example \citet{Latz,Jin1,Jin2} and \citet{MR3732944,MR4119247} for more details on SGD.
These methods have laid the foundations for a variety of extensions and applications, including among many others, fractional differential equations~(\citet{Pang2019}), variational PINNs~(\citet{kharazmi2021variational}), Bayesian variants~(\citet{yang2021b,Jin3}) and mean-field games~(\citet{10572276}); see also \citet{latz1} and \citet{SIKORA2024102340} for a comparison of PINNs with finite element methods.

On the other hand, deep gradient flow methods (DGFMs), also known as deep Ritz methods, formulate the PDE as an energy minimization problem, where the (Dirichlet) energy is derived from the differential operator, which typically leads to a loss function that is easier to compute.
Moreover, they usually discretize the equation in time and train one network for each time step, instead of using a monolithic space-time discretization; see \textit{e.g.}, \citet{E_Yu_2017,Liao_Ming_2019,georgoulis2023discrete,park2023deep,papapantoleon2024time,Bruna_2024} for differential operators, and \citet{georgoulis2024deep} for an integro-differential operator.
A comprehensive review of deep learning methods for the solution of PDEs (and related BSDEs) appears in the forthcoming book of \citet{jentzen2025mathematicalintroductiondeeplearning}.

In the present article, we are interested in analyzing the error of deep gradient flow methods for the solution of PDEs.
Let us consider the PDE 
\begin{equation}
\label{PDE-gen}
    \begin{split}
        u_t + \mathcal Au & = 0, \quad (t,x)\in[0,T]\times D, \\
        u(0,x)              & = \Phi (x), \quad x\in\partial D,
    \end{split}
\end{equation}
where $\mathcal A$ is a differential operator, $\Phi$ determines the initial condition, $T$ is a (finite) time horizon, and $D\subseteq \mathbb R^d$ is the domain of the PDE.    
DGFMs translate the PDE into an energy minimization problem, which is then computed using stochastic gradient descent or one of its variants. 
The solution computed by DGFMs can be described in the following manner:
\begin{equation}
\label{PDE-MIN}
    u^\star_{\theta, n} = \argmin_{v \in \mathcal{C}^n_{\theta}} \int \ell (v(x)) \mathrm dx,
\end{equation}
where $\mathcal{C}^n_{\theta}$ denotes the space of neural networks with $n$ neurons where $\theta$ is the set of trainable parameters, while $\ell$ denotes the loss functional associated to the Dirichlet energy of the operator $\mathcal{A}$.

Let $u^\star$ denote the unique solution of~\eqref{PDE-gen}.
We would like to analyze and study the difference between the true solution of~\eqref{PDE-gen} and the solution computed by the deep gradient flow methods, \emph{i.e.} by the outcome of the minimization problem~\eqref{PDE-MIN}.
This difference is known as the \textit{generalization error} in the machine learning literature, \emph{i.e.} 
\[
    \mathcal E_\text{gen} = \| u^\star - u^\star_{\theta, n}\|.
\]
The generalization error $\mathcal E_\text{gen}$ can be decomposed in four separate components:
\begin{itemize}
    \item the \textit{quadrature error} $\mathcal E_\text{quad}$, which refers to how well the integral in~\eqref{PDE-MIN} is approximated by Monte Carlo simulations or another quadrature method;
    \item the \textit{time-stepping error} $\mathcal E_\text{step}$, which refers to how well the time-stepping scheme is approximating the true PDE;
    \item the \textit{approximation error} $\mathcal E_\text{approx}$, which refers to how well the neural network $v$ can approximate the continuous function $u$ that solves the PDE~\eqref{PDE-gen};
    \item the \textit{training error} $\mathcal E_\text{train}$, which refers to how well SGD approximates the true solution of the minimization problem~\eqref{PDE-MIN}.
\end{itemize}
Then, we have the error decomposition 
\begin{align}
    \label{eq:error-decomposition}
    \mathcal E_\text{gen} = \mathcal E_\text{quad} + \mathcal E_\text{step} + \mathcal E_\text{approx} + \mathcal E_\text{train} ,
\end{align}
and the aim of the present paper is to study these errors and show that, as the number of quadrature points, time-steps, neurons and training time tends to infinity, then the generalization error tends to zero, and the outcome of the deep gradient flow method indeed approximates the solution of the PDE.

There are several articles available that study the generalization error of deep learning methods for PDEs, typically focusing on the popular DGM and PINN methods. 
These methods rely on approximability properties of neural networks and properties of quadrature methods in order to control the generalization error, while they typically consider only \textit{a posteriori} estimates for the training error. 
We refer the interested reader to \citet{Mishra2022PINNs} and \citet{Gazoulis2025} for results on PINNs, and the related while more general article of \citet{loulakis2023newapproachgeneralisationerror}.
Moreover, several articles consider the approximation error of DMG and PINNs; see, for example, \citet{sirignano2018dgm,Shin_Darbon_Karniadakis} and \citet{abdo2026convergenceanalysispinnsfractional}.
The recent article of \citet{jiang2023global} considers the ``global'' convergence of DGM and PINNs, which amounts to the convergence of the training error in our notation. 
Combined with other available results, this article allows to deduce the convergence of the generalization error of the DGM and PINN methods.
Compared to the extensive literature on DGM and PINNs, there are significantly fewer papers on the convergence analysis of DGFMs; let us mention here the articles of \citet{Dondl2022} which focuses on the approximation error, and \citet{doi:10.1142/S021953052350015X} which provides a convergence rate using the Rademacher complexities.

The aim of the present article is to provide convergence results on the generalization error of DGFMs under reasonable and verifiable hypotheses on the underlying PDEs.
The first part of this work focuses on the analysis of the approximation error, \textit{i.e.} we show that there exists a neural network that approximates the solution of the PDE.
This result uses ideas from PDE theory, optimization and the calculus of variations, and a suitable variant of the universal approximation theorem for neural networks.
This part is inspired by the seminal paper of \citet{sirignano2018dgm}, although the methods are rather different.
The second part of this work focuses on the analysis of the training error, and we show that as the number of neurons tend to infinity and the training time also tends to infinity, then the outcome of the deep gradient flow method tends to the true solution of the PDE. 
This result is based on the careful analysis of the gradient flow associated with the loss function of DGFMs.
This part follows the work of \citet{jiang2023global}, utilizing the structure of DGFMs.
The quadrature error and the discretization error of the time-stepping scheme are the most well-understood errors in the error decomposition \eqref{eq:error-decomposition}, thus this work focuses on the other two errors.
The combination of these results, yields that the generalization error of DGFMs, under reasonable and verifiable assumptions, tends to zero.

This article is organized as follows: 
\cref{sec:DGFMs} provides an overview of deep gradient flow methods for the solution of PDEs.
\cref{sec:convergence-NN} studies the approximation error of DGFMs, using the variational formulation of PDEs and a tailored version of the universal approximation theorem. 
\cref{sec:training} studies the training error of DGFMs; we first derive a gradient flow that the training process satisfies in the ``wide network limit'' and then analyze the behavior of this flow as the training time tends to infinity.
Finally, the appendices contain auxiliary estimates and examples.


\subsection{Notation}
Let $\cH$ denote an arbitrary space, then $\|\cdot\|_\mathcal{H}$ denotes the norm, $\langle\cdot,\cdot\rangle_\mathcal{H}$ denotes the inner product, and  $w_m \xrightharpoonup{\,\,\, \mathcal H \,\,\,} w$ denotes the weak convergence on this space.
We abbreviate spaces and norms as $\cH = \cH(\R^d)$ and $\norm{f}_{\cH} = \norm{f}_{\cH (\R^d)}$.

Let $1 \leq p < \infty$ and denote by $L^p (\R^d)$ the space of functions with finite $p$-norm, where
\[
\norm{f}_{L^p} = \left( \int_{\R^d} \abs{f (x)}^p \ud x \right)^{\frac{1}{p}},
\]
while $L^p_{\text{loc}}$ denotes the space of functions in $L^p$ that are locally integrable.
Let $C_c^k \left( \mathbb{R}^d \right)$ denote the space of functions with compact support and continuous partial derivatives up to order $k$.
Moreover, let  $W_0^{k, p} (\R^d)$ denote the Sobolev space with norm
\[
\norm{f}_{W_0^{k, p}} = \left( \sum_{\abs{\alpha} \leq k} \int_{\R^d} \abs{D^\alpha f (x)}^p \ud x \right)^{\frac{1}{p}} < \infty,
\]
with $D^{\alpha} f$ the weak derivative of $f$ and $\alpha$ a multi-index.
Let us introduce the shorthand notation $\cH_0^k (\R^d) := W_0^{k, 2} (\R^d)$ for Sobolev spaces, and let $\cH^{-1}(\R^d)$ denote the dual space of $\cH_0^1(\R^d)$.

Finally, let $\mathcal V \subset \mathcal H \subset \mathcal V^*$ denote a Gelfand triple, in which $\mathcal{H}$ is a separable Hilbert space, $\mathcal{V}$ is a Banach space and $\mathcal{V}^*$ is the topological dual of $\mathcal{V}$.

\begin{definition}[Self-adjoint operator]
An operator $\cL:\mathcal{V} \to \mathcal{V}^*$ is self-adjoint if
\[
    \ip{\cL u,v}_{\cV^*,\cV} = \ip{\cL v,u}_{\cV^*,\cV} \quad \text{for all \ } u,v \in \mathcal{V}.
\]
\end{definition}

\begin{remark}
The inner product $\ip{\cL u,v}_{\cV^*,\cV}$ means that $\cL u$ acts on $v$ as a functional. 
An important example is the following: $\mathcal{V}=\cH_0^1(\R^d)$, $\mathcal{V}^*=\cH^{-1}(\R^d)$, $\mathcal H = L^2(\R^d)$, and $\cL=-\Delta$, where $\Delta$ denotes the Laplace operator. 
Then, we define the functional $\cL u$ as follows 
\begin{align*}
    \ip{\cL u,v}_{\cH^{-1},\cH_0^1}=\ip{-\Delta u,v}_{\cH^{-1},\cH_0^1}:= \ip{\nabla u,\nabla v}_{L^2}.
\end{align*}
\end{remark}

%
%


\section{Deep gradient flow methods for PDEs}
\label{sec:DGFMs}

Let us start by providing an overview of deep gradient flow methods (DGFMs) for the solution of PDEs. 
These methods have gained increased popularity in the literature because they can efficiently handle high-dimensional PDEs stemming from physics, engineering, and finance; see \textit{e.g.} \citet{E_Yu_2017}, \citet{Liao_Ming_2019}, \citet{georgoulis2023discrete}, \citet{park2023deep} and \citet{papapantoleon2024time} for differential operators, and \citet{georgoulis2024deep} for an integro-differential operator.
Deep gradient flow methods reformulate the PDE as an energy minimization problem, which is then approximated in a time-stepping fashion by deep artificial neural networks. 
This method results in a loss function that is tailor-made to the PDE at hand, avoids the use of a second derivative, which is computationally costly, and reduces the training time compared to, for instance, the DGM of \citet{sirignano2018dgm}; see \textit{e.g.} \citet[Sec.~5]{georgoulis2023discrete}. 

Let $u \left( t, x \right) : \left[ 0, T \right] \times \mathbb{R}^{d} \to \mathbb{R}$ be the solution of the following partial (integro-)differential equation:
\begin{equation}
\label{eq:general_operator}
\begin{aligned}
u_t + \mathcal{A} u & = 0, \quad u(0) = u_0, 
\end{aligned}
\end{equation}
where $\mathcal{A}$ is an operator from $\mathcal V$ to $\mathcal V^*$ and $u_0\in\mathcal H$ is the initial condition.  
In order to write the PDE as an energy minimization problem, we need to split the operator in a symmetric and an (asymmetric) remainder part, \emph{i.e.}
\begin{equation}
\label{eq:operator}
\begin{aligned}
\mathcal{A} u & = \mathcal{L} u + F(u), 
\end{aligned}
\end{equation}
where $\mathcal{L}$ is a self-adjoint, linear operator and $F$ is a (linear) operator from $\mathcal{V}$ to $\mathcal V^*$. 
This PDE is then discretized using, for example, the backward Euler differentiation scheme, which yields
\[
\frac{U^k - U^{k - 1}}{h} + \mathcal{L} U^k + F \Big( U^{k - 1} \Big) = 0, \quad U^0  = u_0,
\]
where $U^k$ denotes the approximation to the solution of the PDE $u(t_k)$ at time step $t_k$, on an appropriate grid.
The variational formulation of this equation yields an energy functional $I^k (v)$ such that $U^k$ is a critical point of $I^k$, where
\[
I^k (v) = \frac{1}{2} \norm{v - U^{k - 1}}_{\mathcal{H}}^2 + \frac{h}{2} \ip{\cL v, v}_{\mathcal{V}^*,\mathcal{V}}   + h\ip{F \left( U^{k - 1} \right),v}_{\mathcal H}.
\]
The function $v$ is approximated by artificial neural networks which are trained using the stochastic gradient descent (SGD) algorithm, or one of its variants, while the functional $I^k$ provides a loss function for the SGD iterations which is tailor-made for this problem.
The aim of this paper is to show that this procedure converges to the true solution $u^\star$ of the PDE \eqref{eq:general_operator}.

Let us mention that the convergence results hold for certain non-linear operators $F$, assuming that the PDE is well-posed.
Next, we present examples of PDEs that have been treated by DGFMs, and their applications.

\begin{example}[Heat equation]
The simplest example that fits this framework is the celebrated heat equation, which reads
\[
    u_t = \kappa \Delta u,\quad \kappa>0,
\]
subject to an initial condition.
Then $\mathcal{A} = \cL = - \kappa \Delta$ and $F(u) = 0$.
\end{example}

\begin{example}
\citet{georgoulis2023discrete} consider dissipative evolution PDEs of the following form
\[
    u_t - \nabla \cdot (A \nabla u) = F,
\]
subject to appropriate initial and terminal conditions, where $A$ is a symmetric, uniformly positive definite and bounded diffusion tensor and $F$ is a suitable function.
Then, we have that $\mathcal{A} = \mathcal{L} = - \nabla \cdot (A \nabla u)$.
\end{example}

\begin{example}[Option pricing PDEs]
\label{ex:option}
PDEs arising in the valuation of financial derivatives fit naturally in this setting. 
In the \citet{scholes1973pricing} model, for example, we have directly that
\[
    \mathcal{L} u = - \frac{\sigma^2}{2} \Delta u + r u \quad \text{ and } \quad F(u) = \Big( \frac{\sigma^2}{2} - r \Big) \nabla u.
\]
Here $r$ and $\sigma$ are positive parameters that denote the risk-free interest rate and the asset volatility respectively.

More general and more realistic diffusion models also fit in this framework.
Let us consider the \citet{heston1993closed} model as an example, where $S$ denotes the asset price process and $V$ the variance process. 
The option pricing PDE in this model takes the form \eqref{eq:operator} with 
\begin{equation}\label{eq:PR}
\mathcal{L} u = - \nabla \cdot (A \nabla u ) + ru \quad \text{ and } \quad F(u) = \mathbf{b} \cdot \nabla u,
\end{equation}
where
\[
A = \frac{V}{2} \begin{bmatrix} S^2 & \eta \rho S \\ \eta \rho S & \eta^2 \end{bmatrix} 
    \quad \text{ and } \quad
\mathbf{b} = \begin{bmatrix} (V - r + \frac{1}{2} \rho \eta) S \\ \kappa (V - \theta) + \frac{1}{2} \eta \rho V + \frac{\eta^2}{2} \end{bmatrix}.
\]
Here, $\eta$ denotes the volatility of the volatility, $\rho$ the correlation between the Brownian motions driving the asset price and the variance process, $\theta$ the long term variance and $\kappa$ the reversion rate of the variance to $\theta$.

\end{example}

\begin{example}[Option pricing PIDEs]
\label{ex:op-PIDE}
Certain classes of partial integro-differential equations (PIDEs) arising in the pricing of financial derivatives can also be casted in this framework, in particular when the integro-differential operator is not ``symmetrized''.
Let us consider, for example, the multi-dimensional Merton model as described in \citet{georgoulis2024deep}.
Then, the PIDE arising for the pricing of basket options can be described using \eqref{eq:PR}, where the operator $\mathcal{L}$ retains the same structure, while the function $F$ takes now the form
\[
    F(u) = \mathbf{b} \cdot \nabla u - \lambda \int_{\R^d} \big( u \left( x \ue^z \right) - u (x) \big) \nu(\ud z),
\]
where $\nu$ denotes the multivariate normal density function.
\end{example}

\begin{example}[Allen--Cahn equation]
\citet{park2023deep} consider the example of the two-dimensional Allen--Cahn equation:
\[
\begin{aligned}
u_t  = \Delta u - \epsilon^{-2} W'(u), 
\end{aligned}
\]
with appropriate initial and boundary conditions, where $W$ is a double well potential; for instance, $W(u) = \frac{(u^2 -1 )^2}{4}$. 
Then $\mathcal{L}u = - \Delta u + \epsilon^{-2} W'(u)$ and $F(u) = 0$.
\end{example}

\section{Convergence of the approximation error}
\label{sec:convergence-NN}

In this section, we show that the approximation error of the deep gradient flow method converges to zero, \emph{i.e.} we consider a neural network with a single layer and prove that as the number of nodes in the network tends to infinity, there exists a neural network that converges to the solution of the PDE. This proof consists of several steps. 
First, we show that the problem is well-posed in \cref{sec:well}. 
Second, we prove convergence of the time-stepping scheme in \cref{sec:time}. 
Third, we prove the equivalence between the discretized PDE and the minimization of the variational formulation in \cref{sec:weak}. 
Fourth, we prove a version of the universal approximation theorem (UAT) in \cref{sec:neural}. 
Finally, in \cref{sec:convergence}, we deduce the convergence of the neural network approximation to the solution of the minimization problem by utilizing the UAT. 

In the sequel, we consider the following Gelfand triple: $\mathcal{V}=\cH_0^1(\R^d)$, $\mathcal{V}^*=\cH^{-1}(\R^d)$ and $\mathcal{H}=L^2(\R^d)$.
Let us consider the PDE \eqref{eq:general_operator}--\eqref{eq:operator} and assume that the operators $\cL$ and $F$ satisfy the following conditions.

\begin{namedassumption}{(CON)}
\label{ass:bound} 
Assume that the operators $\cL$ and $F$ satisfy the following inequalities, for any $u,v\in \cH_0^1(\R^d)$,
\begin{align*}
\abs{\ip{\cL  u,  v}_{\cH^{-1},\cH_0^1}} \le M \norm{u}_{\cH_0^1 } \norm{v}_{\cH_0^1 }
    \quad \text{and} \quad
\norm{F(u)}_{L^2}\le M \norm{u}_{\cH_0^1 },
\end{align*}
where $M>0$ is a constant.
\end{namedassumption}

\begin{namedassumption}{(G\AA)}
\label{ass:ellip}
The operator $\cL$ satisfies the G\aa rding inequality, \emph{i.e.} there exist constants $\lambda_1>0,\lambda_2\ge 0$ such that, for any $u\in \cH_0^1(\R^d),$ holds
\begin{align*}
\ip{\cL  u,  u}_{\cH^{-1},\cH_0^1} \ge  \lambda_1\norm{u}_{\cH_0^1}^2 - \lambda_2 \norm{u}^2_{L^2}.
\end{align*}
\end{namedassumption}

\begin{namedassumption}{(SA)}
\label{ass:SA-PD}
The operator $\mathcal L$ is self-adjoint and positive definite.
\end{namedassumption}

\begin{namedassumption}{(LIP)}
\label{ass:Lip}
The operator $F$ satisfies an estimate of the form 
\[
\norm{F(v) - F(w)}_{\mathcal{H}^{-1}} \le \lambda \norm{v - w}_{\mathcal{H}^1_0} + \mu \norm{v - w}_{L^2}, 
\]
for all $v,w \in \left\{ v \in \mathcal{H}_0^1: \min_x \norm{u(x) - v}_{\mathcal{H}^1_0} \leq 1 \right\}$, where $\lambda<1$ and $\mu\in\R$.
\end{namedassumption}

\begin{remark}
The examples of PDEs considered in the previous section typically satisfy these assumptions. 
More details, focusing on the option pricing PDEs of \cref{ex:option,ex:op-PIDE}, are deferred to \cref{appendix:examples}.
\end{remark}


\subsection{Well-posedness}
\label{sec:well}

Let us first discuss the existence and uniqueness of solutions for equation \eqref{eq:general_operator}.

\begin{theorem}[Well-posedness]
Assume that the operators $\cL$ and $F$ satisfy \cref{ass:bound,ass:ellip}, then equation \eqref{eq:general_operator} admits a unique weak solution $u \in L^2 \left( \left( 0, T \right) ; \cH_0^1 (\R^d) \right) \cap \cH^1 \left( \left( 0, T \right) ; \cH^{-1} (\R^d) \right)$, that satisfies
\begin{align*}
\frac{\ud}{\ud t}\ip{u, v}_{L^2} + \ip{\cL  u,  v}_{\cH^{-1},\cH_0^1}  + \ip{F(u), v}_{L^2} = 0
\end{align*}
for any $v \in \cH_0^1 (\R^d)$ and $u \left( 0 \right) = u_0$. 
\end{theorem}

\begin{proof}
According to \citet[Theorem 3.2.2]{hilber2013computational}, we only need to verify that the bilinear form $\ip{\mathcal{A} u, v}_{\cH^{-1},\cH_0^1}$ is continuous and satisfies the “G\aa rding inequality”, where 
\[
\ip{\mathcal{A} u, v}_{\cH^{-1},\cH_0^1}= \ip{\cL u, v}_{\cH^{-1},\cH_0^1} + \ip{F(u), v}_{L^2}.
\]
The continuity follows directly from \cref{ass:bound} and the Cauchy--Schwarz inequality, since
\begin{align*}
\abs{\ip{\mathcal{A} u, v}_{\cH^{-1},\cH_0^1}}
    & \le \abs{\ip{\cL u,v}_{\cH^{-1},\cH_0^1}}  + \abs{\ip{F(u), v}_{L^2}} \\ 
    & \le M \left[ \norm{u}_{\cH_0^1} \norm{ v}_{\cH_0^1} + \norm{u}_{\cH_0^1} \norm{v}_{L^2} \right] 
      \le 2M \norm{u}_{\cH_0^1} \norm{v}_{\cH_0^1}.
\end{align*}
Let us also verify that the bilinear form satisfies the G\aa rding inequality, \emph{i.e.} that there exist $C_1, C_2 > 0$, such that
\[
\abs{\ip{\mathcal{A} u, u}_{\cH^{-1},\cH_0^1}} \ge C_1 \norm{u}_{\cH_0^1}^2 - C_2 \norm{u}_{L^2}^2.
\]
We have that 
\begin{align*}
\abs{\ip{\mathcal{A} u, u}_{\cH^{-1},\cH_0^1}}
    & \ge \abs{\ip{\mathcal{L} u, u}_{\cH^{-1},\cH_0^1}} - \abs{\ip{F(u), u}_{L^2}} \\ 
    & \ge \lambda_1 \norm{u}_{\cH_0^1}^2 - \lambda_2 \norm{u}^2_{L^2} - M \norm{u}_{\cH_0^1} \norm{u}_{L^2}  \\ 
    & \ge \lambda_1 \norm{u}_{\cH_0^1}^2 - \lambda_2 \norm{u}^2_{L^2} - M \left(  \frac{\lambda_1}{2M} \norm{u}_{\cH_0^1}^2 + \frac{M}{2 \lambda_1} \norm{u}_{L^2}^2 \right) \\ 
    & = \frac{\lambda_1}{2} \norm{u}_{\cH_0^1}^2 - \left( \lambda_2 + \frac{M^2}{2 \lambda_1} \right) \norm{u}^2_{L^2},
\end{align*}
where have used \cref{ass:bound,ass:ellip} and the Cauchy--Schwarz inequality for the second step, and the Young inequality with $\varepsilon = \frac{\lambda_1}{M}$ for the third step.
\end{proof}


\subsection{Time stepping}
\label{sec:time}

The second step is to discretize the PDE in time and prove that this discretization converges to the true solution as the time step tends to zero.
Consider the PDE in formulation \eqref{eq:general_operator}--\eqref{eq:operator}, \textit{\emph{i.e.}}
\[
\begin{aligned}
u_t + \mathcal{L} u + F (u) = 0, \quad u(0) = u_0.
\end{aligned}
\]
Let us divide $[0,T]$ in $K$ intervals $(t_{k - 1}, t_k]$ with step size $h = t_k - t_{k - 1} = \frac{1}{K}$. 
Let $U^k$ denote the approximation of $u (t_k)$ using the backward Euler discretization scheme, \emph{i.e.}
\begin{equation}
\label{eq:discretization}
\frac{U^k - U^{k - 1}}{h} + \mathcal{L} U^k + F \left( U^{k - 1} \right) = 0, \quad U^0 = u_0.
\end{equation}

\begin{theorem}
\label{thm:time-stepping}
Assume that the operators $\mathcal{L}$ and $F$ satisfy \cref{ass:bound,ass:ellip,ass:SA-PD,ass:Lip}.
Then, there exists a constant $C$ independent of $h$ and $k$ such that, for $h$ sufficiently small, holds
\[
 \max_{0 \leq k \leq K} \norm{u(t_k) - U^k}_{L^2} \leq C h.
\]
\end{theorem}

\begin{proof}
The proof follows directly from Theorem 2.1 in \citet{Akrivis_Crouzeix_Makridakis_1999}.
Indeed, using that $U^0 = u \left( 0 \right)$, we can show by direct, but tedious, calculations that the assumptions of \cite[p.~523]{Akrivis_Crouzeix_Makridakis_1999} are satisfied for $\lambda<1$ and $q=1$.
\end{proof}

\subsection{Weak formulation and uniqueness of minimizer}
\label{sec:weak}


The third step is to reformulate equation \eqref{eq:discretization} as a variational problem and prove that its solution is equivalent to the minimization of an energy functional. 
Let us first rewrite \eqref{eq:discretization} as follows
\begin{equation}
\label{eq:discretization2}
\left( U^k - U^{k - 1} \right) + h \left( \mathcal{L} U^k + F \left( U^{k - 1} \right) \right) = 0, \quad U^0  = u_0.
\end{equation}
We want to find an energy functional $I^k (u)$ such that $U^k$ is a critical point of $I^{k}$. 
Consider the following functional $I^k$ on $\mathcal{H}_0^1 (\R^d)$
\begin{align}
\label{eq:energy_functional}
I^k (u) 
    & = \frac{1}{2} \norm{u - U^{k - 1}}^2_{L^2} + \frac{h}{2} \ip{\cL  u,  u}_{\cH^{-1},\cH_0^1}   + h\ip{F \left( U^{k - 1} \right),u}_{L^2} \\ 
    & =: \mathcal{M}^k (u) + \mathcal{G}^k (u), \nonumber\\
\intertext{where} 
\mathcal{M}^k (u) & = \frac{1}{2} \norm{u}^2_{L^2} + \frac{h}{2} \ip{\cL  u,  u}_{\cH^{-1},\cH_0^1} \nonumber \\ 
\intertext{and} 
\mathcal{G}^k (u) & = - \ip{u, U^{k - 1}}_{L^2} + \frac{1}{2} \norm{U^{k - 1}}^2_{L^2} + h \ip{F \left( U^{k - 1} \right),u}_{L^2}. \nonumber
\end{align}
Here, $\mathcal{G}^k$ is a linear functional and $\mathcal{M}^k$ is a nonlinear (quadratic) term.

\begin{theorem}
\label{thm:var_to_pde}
Assume that the operators $\mathcal{L}$ and $F$ satisfy \cref{ass:bound,ass:ellip,ass:SA-PD,ass:Lip} and that $0<h<\frac{1}{2\lambda_2},$ where $\lambda_2$ is the constant from \cref{ass:ellip}. 
Then, the minimizer of \eqref{eq:energy_functional} is the unique solution of \eqref{eq:discretization2} in $\mathcal{H}_0^1 (\R^d)$.
\end{theorem}

The proof of this theorem is based on the following two preparatory results.

\begin{lemma}
\label{lem:weak_con}
Consider the setting of \cref{thm:var_to_pde}. 
Then, the functional $I^k$ is bounded from below and, for any $w_*\in \cH_0^1(\R^d)$ and sequence $w_m \xrightharpoonup{\mathcal{H}_0^1} w_*$, we have
\begin{align*}
    \liminf_{m \to \infty} I^k \left( w_m \right) \ge I^k \left( w_* \right).
\end{align*}
\end{lemma}

\begin{proof}
Let us first prove that $I^k$ is bounded from below. 
Using the Cauchy--Schwarz inequality and the inequality $\alpha\beta\leq \alpha^2/4+\beta^2$, we get that
\begin{align*}
\mathcal{G}^k (u) 
    & \ge - \norm{U^{k - 1}}_{L^2} \norm{u}_{L^2} + \frac{1}{2} \norm{U^{k - 1}}^2_{L^2} - h \norm{F \left( U^{k - 1} \right)}_{L^2} \norm{u}_{L^2} \\ 
    &= - \norm{u}_{L^2}\Big(\norm{U^{k - 1}}_{L^2}+ h \norm{F \left( U^{k - 1} \right)}_{L^2}\Big) + \frac{1}{2} \norm{U^{k - 1}}^2_{L^2}  \\
    &\ge - \frac{1}{4} \norm{u}^2_{L^2} - \left\{ \norm{U^{k - 1}}_{L^2} + h \norm{F \left( U^{k - 1} \right)}_{L^2} \right\} ^2 + \frac{1}{2} \norm{U^{k - 1}}^2_{L^2}.
\end{align*}
Hence, the functional $I^k$ satisfies
\begin{align*}
I^k (u) 
    &= \mathcal{M}^k (u) + \mathcal{G}^k (u) \\ 
    &\ge \frac{1}{4} \norm{u}^2_{L^2} + \frac{h}{2} \ip{\cL  u,  u}_{\cH^{-1},\cH_0^1} - \left\{ \norm{U^{k - 1}}_{L^2} + h \norm{F \left( U^{k - 1} \right)}_{L^2} \right\} ^2 + \frac{1}{2} \norm{U^{k - 1}}^2_{L^2}.
\end{align*}
Using \cref{ass:ellip} and the condition $h<\frac{1}{2\lambda_2}$, we have that $I^k (u)$ is bounded below by 
\begin{align}\label{eq:Ik-bound}
I^k (u) 
    &\ge \frac{1}{4} \norm{u}^2_{L^2} + \frac{h}{2} \ip{\cL  u,  u}_{\cH^{-1},\cH_0^1} - R^{\left( 1 \right)}_k \nonumber\\
    &\ge  \Big(\frac{1}{4}-\frac{h \lambda_2}{2}\Big) \norm{u}^2_{L^2} + \frac{h\lambda_1}{2} \norm{u}^2_{\cH_0^1} - R^{\left( 1 \right)}_k,
\end{align}
where the term $R_k^{(1)}$, defined below, is independent of $u$ and finite
\begin{align}\label{eq:jnw}
R^{\left( 1 \right)}_k := \left\{ \norm{U^{k - 1}}_{L^2} + h\norm{F \left( U^{k - 1} \right)}_{L^2} \right\} ^2 - \frac{1}{2} \norm{U^{k - 1}}^2_{L^2}. 
\end{align}

As for the second part, consider $w_*\in \cH_0^1(\R^d)$ and a sequence $(w_m)_m$ such that $w_m \xrightharpoonup{\mathcal{H}_0^1} w_*$ as $m\to\infty$.
Then, by the definition of weak convergence
\begin{align*}
\frac{1}{2} \ip{w_m, w_*}_{L^2} + \frac{h}{2} \ip{\cL  w_*, w_m}_{\cH^{-1},\cH_0^1} &\xrightarrow[\ m\to\infty\ ]{} \frac{1}{2} \norm{w_*}^2_{L^2} + \frac{h}{2}  \ip{\cL  w_*,  w_*}_{\cH^{-1},\cH_0^1}, \\ 
\intertext{while for the linear part we also have that} 
\mathcal{G}^k \left( w_m \right) &\xrightarrow[\ m\to\infty\ ]{} \mathcal{G}^k \left( w_* \right) .
\end{align*}
Consider now the functional
\begin{align}\label{eq:diff-Ik}
I^k \left( w_m \right) + I^k \left( w_* \right) - I^k \left( w_* - w_m \right) 
    = \underbrace{\ip{w_m, w_*}_{L^2} + h \ip{\cL  w_*,  w_m}_{\cH^{-1},\cH_0^1} + 2 \mathcal{G}^k \left( w_m \right)}_{\longrightarrow \ 2 I^k \left( w_* \right)},
\end{align}
and notice that 
\begin{align}\label{eq:Mk-positive}
\mathcal{M}^k (u)  = \frac{1}{2} \norm{u}^2_{L^2} + \frac{h}{2} \ip{\cL  u,  u}_{\cH^{-1},\cH_0^1}\ge \Big(\frac{1}{2}-\frac{h \lambda_2}{2}\Big) \norm{u}^2_{L^2} + \frac{h\lambda_1}{2} \norm{u}^2_{\cH_0^1}\ge 0,
\end{align}
from \cref{ass:ellip} and $h<\frac{1}{2\lambda_2}$.
Then, taking the limit as $m \to \infty$ on both sides of \eqref{eq:diff-Ik} and using \eqref{eq:Mk-positive}, we get that
\begin{align*}
\liminf_{m \to \infty} I^k \left( w_m \right) + I^k \left( w_* \right) - \underbrace{\liminf_{m \to \infty} I^k \left( w_* - w_m \right)}_{ \ge 0} \ge 2 I^k \left( w_* \right),
\end{align*}
which implies $\liminf_{m \to \infty} I^k \left( w_m \right) \ge I^k \left( w_* \right)$. 
\end{proof}

\begin{proposition}
\label{prop:main1}
Consider the setting of \cref{thm:var_to_pde}. 
Let $U^{k - 1} \in \mathcal{H}_0^1 (\R^d)$, then there exists a unique minimizer in $\mathcal{H}_0^1 (\R^d)$ of the functional $I^k$. 
\end{proposition}

\begin{proof}
Let us first show the uniqueness of the minimizer of the functional $I^k$. 
Let $w_1,w_2\in\mathcal{H}_0^1 (\R^d)$ be two minimizers of $I^k$ then, using \Cref{ass:SA-PD,ass:ellip}, we get
\begin{align*}
I^k \left( w_1 \right) + I^k \left( w_2 \right) - 2 I^k \left( \frac{w_1 + w_2}{2} \right) 
    &= \frac{1}{4}\norm{w_1 - w_2}^2_{L^2} + \frac{h}{4} \ip{\cL  (w_1-w_2),  w_1-w_2}_{\cH^{-1},\cH_0^1} \\
    &\ge \Big(\frac{1}{4}-\frac{h \lambda_2}{4}\Big) \norm{w_1-w_2}^2_{L^2} + \frac{h\lambda_1}{4} \norm{w_1-w_2}^2_{\cH_0^1}
    \overset{\eqref{eq:Mk-positive}}{\ge} 0, 
\end{align*}
which is $0$ if and only if $w_1 = w_2$ almost everywhere. Otherwise, $I^k \left( \frac{w_1 + w_2}{2} \right)$ is smaller than $I^k \left( w_1 \right)$, which is a contradiction.

Next, we show the existence of a minimizer for $I^k$.
Define the bounded set $\mathcal{B}^k \subset \mathcal{H}_0^1 (\R^d)$ via
\begin{align*}
\mathcal{B}^k 
    := \left\{ f \in \mathcal{H}_0^1 (\R^d) \Big| \ \Big(\frac{1}{4}-\frac{h \lambda_2}{2}\Big) \norm{f}^2_{L^2} + \frac{h \lambda_1}{2} \norm{f}^2_{\cH_0^1}  \leq R^{\left( 1 \right)}_k +\frac{1}{2}\norm{U^{k-1}}^2_{L^2}\right\} , 
\end{align*}
where the constant $R^{\left( 1 \right)}_k$ is defined in \eqref{eq:jnw}. 
Consider an $f \notin \mathcal{B}^k$ then, using inequality \eqref{eq:Ik-bound}, we have that $I^k \left( f \right) \ge \frac{1}{2}\norm{U^{k-1}}^2_{L^2}$. 
Using that $0\in \mathcal{B}^k$, $I^k( 0 ) = \frac{1}{2}\norm{U^{k-1}}^2_{L^2}$, and that $I^k$ is bounded from below, we conclude that
\begin{align*}
    \inf_{w \in \mathcal{B}^k} I^k \left( w \right) = \inf_{w \in \mathcal{H}_0^1} I^k \left( w \right) > - \infty.
\end{align*}
Let us now choose $w_m \in \mathcal{B}^k $ such that $I^k \left( w_m \right) \to \inf_{w \in \mathcal{B}^k} I^k \left( w \right)$. 
Let us also define $w_*$ as the weak limit of $w_m$ in $\cH_0^1$. 
Then, by \cref{lem:weak_con}, 
\begin{align*}
 \inf_{w \in \mathcal{H}_0^1} I^k \left( w \right) = \inf_{w \in \mathcal{B}^k} I^k \left( w \right) = \liminf_{m \to \infty} I^k \left( w_m \right) \ge I^k \left( w_* \right) .
\end{align*}
The last inequality readily implies $I^k \left( w_* \right) = \inf_{w \in \mathcal{H}_0^1} I^k \left( w \right)$.
\end{proof}

\begin{proof}[Proof of \cref{thm:var_to_pde}]
Consider the homogeneous equation 
\[
 \frac{w}{h} +\cL w = 0.
\]
Multiplying with $w$ on each side and integrating, implies $\frac{1}{h}\norm{w}^2_{L^2} + \ip{\cL  w,  w}_{\cH^{-1},\cH_0^1}  = 0$. 
Using \cref{ass:ellip} and $h< \frac{1}{2\lambda_2} $, yields that $w=0.$ 
Therefore the homogeneous equation only has the solution $w=0$ in $\mathcal{H}_0^1 (\R^d)$. 
Thus, the solution of \eqref{eq:discretization2} is unique.
%

Assume that $U^k$ minimizes $I^k$, and let $v$ be a smooth function. 
Consider the function
\begin{align*}
i^k \left( \tau \right) 
     = I^k \left( U^k + \tau v \right) 
    &= \frac{1}{2} \norm{U^k + \tau v - U^{k - 1}}^2_{L^2} \\
    &\quad + \frac{h}{2} \ip{\cL  (U^k + \tau v),  U^k + \tau v}_{\cH^{-1},\cH_0^1}   +h\ip{ F \left( U^{k - 1} \right),U^k+\tau v}_{L^2}, 
\end{align*}
for $\tau \in \mathbb{R}$. 
Since $U^k$ minimizes $I^k$, $\tau = 0$ should minimize $i^k$. 
Hence, we have
\begin{align*}
 0 = \left( i^k \right)' \left( 0 \right)  
 = & \ip{U^k - U^{k - 1},v}_{L^2}+ \frac{h}{2}\Big(\ip{\cL  U^k,  v}_{\cH^{-1},\cH_0^1}+\ip{\cL  v,U^k}_{\cH^{-1},\cH_0^1}\Big)  + h \ip{F \left( U^{k - 1} \right),v}_{L^2} \\ 
 = & \ip{U^k - U^{k - 1},v}_{L^2} + h\ip{\cL  U^k,  v}_{\cH^{-1},\cH_0^1} + h\ip{F \left( U^{k - 1} \right),v}_{L^2}, 
\end{align*}
where in the last equality we used that $\cL$ is self-adjoint. 
This equality must hold for all $v$, thus \eqref{eq:discretization2} holds. 
Finally, note that the second derivative of $i^k$ equals
\[
\left( i^k \right)''(\tau) = \norm{v}_{L^2}^2 + h \ip{\cL v, v}_{\cH^{-1},\cH_0^1}  \ge (1 - h \lambda_2) \norm{v}_{L^2}^2 + h \lambda_1 \norm{v}_{\cH_0^1}^2 > 0,
\]
where we used \cref{ass:ellip}. 
Therefore, $\tau=0$ is indeed the minimizer.
\end{proof}



\subsection{Neural network approximation and a version of the Universal Approximation Theorem}
\label{sec:neural}

We use a neural network to approximate the solution of the PDE \eqref{eq:discretization} or, more specifically, the solution of the optimization problem $\min_{u\in\mathcal{H}^1_0} I^k(u)$ in \eqref{eq:energy_functional}. 
The fourth step is to consider a more general problem: show that any function $v \in \cH_0^1 (\R^d)$ can be approximated by a neural network. 
\citet{hornik1991approximation} proved that a different class of neural networks, see \cref{rem:hornik}, is dense in $\cH_0^1 \left( D\right)$, for some bounded domain $D\subseteq\R^d$. 
However, in our case, the domain equals $\R^d$, therefore we need a tailor-made version of the Universal Approximation Theorem. 

\begin{definition}[Activation function]
\label{def:activation function}
An activation function is a function $\psi: \mathbb{R}^d \to \mathbb{R}$ such that $\psi \in C_c^\infty (\R^d)$ and $\int_{\R^d}\psi (x) \ud x \ne 0$.
\end{definition}

\begin{definition}[Neural network]
\label{def:NN}
Let $\psi$ be an activation function, then we define
\[
    \mathcal{C}^n \left(\psi\right) = \left\{ \zeta: \mathbb{R}^d \to \mathbb{R} \ \big| \ \zeta (x) = \sum_{i = 1}^n \beta_i \psi ( \alpha_i x + c_i ) \right\}, 
\]
as the class of neural networks with a single hidden layer and $n$ hidden units. 
The vector of weights and biases equals
\[
\theta = \left( \beta_1, \dots, \beta_n, \alpha_1, \dots, \alpha_n, c_1, \dots, c_n \right) \in \mathbb{R}^n \times \R^n \times \R^{d \times n},
\]
with $\alpha_i \ne 0$ for all $i\in\{1,\dots,n\}$, thus the dimension of the parameter space equals $\left(2+d\right)n$. 
Moreover, we set $\mathcal{C} \left( \psi \right) = \cup_{n \ge 1} \mathcal{C}^n \left( \psi \right)$. 
\end{definition}

\begin{remark}
In the sequel, we consider PDEs that take values in $\R^d$, thus choosing an activation function $\psi$ in $C_c^\infty (\R^d)$ is convenient. 
Then, we require that $\alpha_i \ne 0$, otherwise $\psi \left( \alpha_ix + c_i \right)$ is a constant, which is not integrable on $\R^d$.
\end{remark}

\begin{remark}\label{rem:hornik}
\citet{hornik1991approximation} introduced a class of neural networks of the form $\xi\left(x\right)=\sum_{i=1}^n \beta_i\phi \left( a_i\cdot x + c_i \right)$ where $\phi: \R \to \R$ and $a_i\in \R^d.$ 
The dimension of the parameter space in this case equals again $\left(2+d\right)n$.
However, this kind of neural network does not belong to $L^2(\R^d)$. 
In fact, it is not possible to prove that this class of neural networks is dense in $\cH_0^1 (\R^d),$ even if $\phi$ has compact support. 
Consider, for example, the case $d=2,$ set $n=1$, $\phi=\boldsymbol{1}_{\left[-1,1\right]},$ $a_1=\left(1,-1\right)$. 
Then $\norm{\phi\left(\alpha_1\cdot\right)}_{L^2}=\int_{\R^2}\boldsymbol{1}_{\abs{x-y}\le 1}\ud x \ud y,$ which is the area of an unbounded belt, and therefore equal to $+\infty.$ 
\end{remark}

\begin{theorem}
    \label{thm:neural app}
Let $\psi$ be an activation function, then the space of neural networks $\mathcal{C}\left(\psi\right)$ is dense in $\mathcal{H}_0^1 (\R^d)$. 
\end{theorem}

The proof of this theorem builds on the proof of the next two lemmata.

\begin{lemma}    \label{lem:1}
Let $\psi$ be an activation function. 
Let $g$ be a continuous function, \emph{i.e.} $g \in C (\R^d)$. 
Suppose that, for any $\zeta \in \mathcal{C}(\psi)$, holds $\int_{\R^d} \zeta (x) g (x) \ud x = 0$. 
Then $g = 0$. 
\end{lemma}

\begin{remark}
Since $\psi \in C^\infty_c (\R^d)$, any function $\zeta$ in $\mathcal{C} \left(\psi\right)$ has compact support. Hence $\int_{\R^d} \zeta (x) g (x) \ud x$ is well-defined. 
\end{remark}

\begin{proof}
Let $g \in C (\R^d)$, $x \in \R^d$, $0 < \e \leq 1$, and define 
\[
\Phi^\e \left( g \right) (x) := \int_{\R^d} \e^{-d} \psi \left( \frac{x - y}{\e} \right) g \left( y \right) \ud y.
\]
We would like to show that 
\[
 \lim_{\e \to 0} \Phi^\e \left( g \right) (x) = cg (x),
\]
where $c = - \int_{\R^d} \psi (x) \ud x$. 
Using a change of variables twice, we have that 
\begin{align*}
\Phi^\e \left( g \right) (x) 
    &= \int_{\R^d} \e^{- d} \psi \left( \frac{x - y}{\e} \right) g \left( y \right) \ud y 
    \stackrel{z = x - y}{=} - \int_{\R^d} \e^{- d}\psi \left( \frac{z}{\e} \right) g \left( x - z \right) \ud z \\ 
    &\stackrel{m = \e^{-1}z}{=} - \int_{\R^d} \psi \left( m \right) g \left( x - \e m \right) \ud m 
    = - \int_K \psi \left( m \right) g \left( x - \e m \right) \ud m,
\end{align*}
where $K$ denotes the (compact) support of $\psi$. 
Notice that, since $z$ is a vector, we have that $m=\frac{z}{\e}$ yields $\ud m = \e^{-d}\ud z$.
Then, using the dominated convergence theorem, we get that
\begin{align*}
\lim_{\e \to 0} \Phi^\e \left( g \right) (x) 
    = \lim_{\e \to 0} \Big\{ - \int_K \psi \left( m \right) g \left( x - \e m \right) \ud m \Big\} = cg (x).
\end{align*}
Now, consider any $\zeta \in \mathcal{C} \left(\psi\right)$ such that $\int_{\R^d} \zeta (y) g (y) \ud y = 0$; then, for any $x \in \R^d$, setting $n=1$, $\beta=\e^{-d}$, $\alpha = \e^{-1}$ and $c = \frac{x}{\e}$ in the definition of $\mathcal{C} \left(\psi\right)$, we get that
\[
 \int_{\R^d} \e^{- d} \psi \left( \frac{x - y}{\e} \right) g \left( y \right) \ud y = 0.
\]
We conclude the proof by sending $\e \to 0$ and using that $c \ne 0$, by definition of an activation function.
\end{proof}


\begin{lemma}
\label{lem:2}
Let $w$ be a function on $C^\infty (\R^d)$ with support on the unit sphere, where
\[
w (x) = \left\{ 
\begin{aligned}
c\exp \left( \frac{-1}{1 - \abs{x}^2} \right), \quad & \text{if $\abs{x} < 1$}, \\ 
0, \quad & \text{if $\abs{x} \ge 1$,}
\end{aligned}
\right.
\]
where $c$ is a constant such that the integral of $w$ equals $1$. 
Let $f \in L^1_{\emph{loc}} (\R^d)$, and introduce
\[
J^\e f (x) = w_\e*f (x) = \int_{\R^d} w_\e \left( y \right) f \left( x - y \right) \ud y, 
\]
with $w_\e = \e^{-d}w \left( \frac{x}{\e} \right)$.
Then, for any $\varphi \in C^\infty_c$ and $f \in L^1_{\emph{loc}} (\R^d)$, we have that
\begin{align*}
 \lim_{\e \to 0}\ip{\varphi, J^\e f}_{L^2} = \ip{\varphi, f}_{L^2}.
\end{align*}
\end{lemma}

\begin{remark}
The convolution of $w_\e$ with $f$ is convenient, because then $J^\e f$ is infinitely differentiable. 
\end{remark}

\begin{proof}
Let us first rewrite $\ip{\varphi, J^\e f}_{L^2}$ as follows
\begin{align*}
\ip{\varphi, J^\e f}_{L^2} 
    &= \int_{\R^d} \int_{\R^d} \varphi (x) w_\e \left( y \right) f \left( x - y \right) \ud y \ud x
    = \e^{-d} \int_{\R^d} \int_{\R^d} \varphi (x) w \left( \frac{y}{\e} \right) f \left( x - y \right) \ud y \ud x \\ 
    &\stackrel{z = y/\e }{=} \int_{\R^d} \int_{\R^d} \varphi (x) w \left( z \right) f \left( x - \e z \right) \ud z \ud x = \int_K \int_K \varphi (x) w \left( z \right) f \left( x - \e z \right) \ud z \ud x
\end{align*}
where $K$ is a compact set that contains the support of $w$ and $\varphi$. 
Hence, by the dominated convergence theorem and Lusin's theorem, letting $\e \to 0$ and using that the integral of $w$ equals 1, we have
\[
 \lim_{\e \to 0}\ip{\varphi, J^\e f}_{L^2} = \lim_{\e \to 0} \int_K \int_K \varphi (x) w \left( z \right) f \left( x - \e z \right) \ud z \ud x = \ip{\varphi, f}_{L^2}. \qedhere
\]
\end{proof}

\begin{proof}[Proof of \cref{thm:neural app}]
Observe that $\mathcal{C} \left(\psi\right) \subset \mathcal{H}_0^1 (\R^d)$, since $\psi \in C^\infty_c (\R^d)$. Assume that $\mathcal{C} \left(\psi\right)$ is not dense in $\mathcal{H}_0^1 (\R^d)$ then, as a corollary of the  Hahn--Banach extension theorem, see \textit{e.g.} \citet[Corollary 4.12]{Jan}, there exists a non-zero continuous linear functional $G$ on $\mathcal{H}_0^1 (\R^d)$ such that for any $\zeta \in \mathcal{C} \left(\psi\right)$,
\begin{align*}
 G \left( \zeta \right) = 0.
\end{align*}
Using the Riesz representation theorem, there exists a $g \ne 0$ in $\mathcal{H}_0^1 (\R^d)$, such that for any $f \in \mathcal{H}_0^1 (\R^d)$,
\[
 \ip{f, g}_{\mathcal{H}_0^1 (\R^d)} = G \left( f \right). 
 \]
Therefore $\ip{\zeta, g}_{L^2} + \ip{\nabla \zeta, \nabla g}_{L^2} = 0$. 
Let us denote $g^\e_1 = J^\e g$ and $g^\e_2 = J^\e \nabla g$, for convenience. 
Consider the inner product of these functions with $\zeta$, then we have
\begin{align*}
\ip{\zeta, g^\e_1}_{L^2} &+ \ip{\nabla \zeta, g^\e_2}_{L^2} \\ 
    &= \int_{\R^d} \int_{\R^d} \zeta (x) w_\e \left( y \right) g \left( x - y \right) \ud y \ud x + \int_{\R^d} \int_{\R^d} \nabla \zeta (x) \cdot \left[ w_\e \left( y \right) \nabla g \left( x - y \right) \right] \ud y \ud x \\ 
    &= \int_{\R^d} \left( \int_{\R^d} \zeta (x) g \left( x - y \right) \ud x \right) w_\e \left( y \right) \ud y + \int_{\R^d} \left( \int_{\R^d} \nabla \zeta (x) \cdot \nabla g \left( x - y \right) \ud x \right) w_\e \left( y \right) \ud y \\ 
    &\stackrel{x = z + y}{=}  \int_{\R^d} \left( \int_{\R^d} \zeta \left( z + y \right) g \left( z \right) dz \right) w_\e \left( y \right) \ud y + \int_{\R^d} \left( \int_{\R^d} \nabla \zeta \left( z + y \right) \cdot \nabla g \left( z \right) dz \right) w_\e \left( y \right) \ud y \\ 
    &=  \int_{\R^d} \Big\{ \ip{\zeta \left( \cdot + y \right), g}_{L^2} + \ip{\nabla \zeta \left( \cdot + y \right), \nabla g}_{L^2} \Big\} \, w_\e \left( y \right) \ud y = 0.
\end{align*}
Since $\zeta \in \mathcal{C} \left(\psi\right)$, it has compact support and we can apply Fubini's theorem in the second step, while we can also use that $\zeta \left( \cdot + y \right) \in \mathcal{C} \left(\psi\right)$ in the last step. 
Hence $\ip{\zeta, g^\e_1}_{L^2} + \ip{\nabla \zeta, g^\e_2}_{L^2} = 0$.
Then, using integration by parts,
\begin{align*}
 \ip{\zeta, g^\e_1 - \nabla \cdot \left( g^\e_2 \right)}_{L^2} = 0.
\end{align*}
Since $g^\e_1 - \nabla \cdot \left( g^\e_2 \right)$ is continuous, see \textit{e.g.} \citet[Proposition 11.1]{Jan}, by \cref{lem:1}, we conclude that $g^\e_1 - \nabla \cdot \left( g^\e_2 \right) = 0$. 
Then, for any $f \in C^\infty_c (\R^d)$, 
\[
 \ip{f, g^\e_1}_{L^2} + \ip{\nabla f, g^\e_2}_{L^2} = \ip{f, g^\e_1 - \nabla \cdot \left( g^\e_2 \right)}_{L^2} = 0. 
\]
Using \cref{lem:2}, for any $f \in C^\infty_c (\R^d)$, we get that
\begin{align*}
 G \left( f \right) = \ip{f, g}_{L^2} + \ip{\nabla f, \nabla g}_{L^2} = \lim_{\e \to 0} \ip{f, g^\e_1}_{L^2} + \ip{\nabla f, g^\e_2}_{L^2} = 0.
\end{align*}
Since $C^\infty_c (\R^d)$ is dense in $\mathcal H_0^1$ with norm $\norm{\cdot}_{\mathcal{H}_0^1}$, $G$ is a zero functional on $\mathcal{H}_0^1$, which is a contradiction.
\end{proof}


\subsection{Convergence of the minimizer}
\label{sec:convergence}

The final step, is to show that the minimizer approximated by neural networks converges to the solution of the PDE, which yields that the approximation error of the method converges to zero.
\cref{thm:var_to_pde} yields that the minimizer of the functional $I^k$ in \eqref{eq:energy_functional} equals the unique solution of discretized equation \eqref{eq:discretization2}.
Here, we show that this minimizer can be approximated by a neural network as defined in \cref{def:NN}.
The final conclusion, \emph{i.e.} the convergence to the true solution of PDE \eqref{eq:general_operator}--\eqref{eq:operator}, follows by an application of \cref{thm:time-stepping}, which shows that the time-stepping scheme converges to the PDE.

\begin{theorem}
\label{thm:5}
Let $(w_m)_{m\in\mathbb N}$ be a sequence in $\mathcal{H}_0^1 (\R^d)$ and $w_*$ be the minimizer of $I^k$. Then 
\[
    \lim_{m \to \infty}\norm{w_m - w_*}_{\mathcal{H}^1_0} = 0 
        \quad \text{ if and only if } \quad 
    \lim_{m \to \infty} I^k \left( w_m \right) = I^k \left( w_* \right).
\]
\end{theorem}

\begin{remark}
Therefore, we can select the approximation sequence $(w_m)$ from the space of neural networks $\mathcal{C} \left(\psi\right)$. 
Using that $\mathcal{C} \left(\psi\right)$ is dense in $\cH_0^1 (\R^d)$, an approximation sequence always exists. 
\end{remark}

\begin{remark}
Let us point out that for an arbitrary $u \in \cH_0^1 (\R^d)$, $\abs{I^k \left( u_m \right) - I^k (u)} \to 0$, does not imply $\norm{u_m - u}_{L^2} \to 0$.
Consider, for example, $F = 0$, then $I^k$ is quadratic and we can always choose $u_m = - u$ since $I^k (u) = I^k \left( - u \right)$. 
\end{remark}


\begin{proposition}[Continuity]
\label{prop:functional_convergence}
Assume that the operators $\mathcal L$ and $F$ satisfy \cref{ass:bound,ass:SA-PD}, then the functional $I^k$ is continuous, \emph{i.e.} for any $f, u \in \cH_0^1 (\R^d)$, holds
\begin{align*}
    \big| I^k \left( f \right) - I^k (u) \big| \le \left( 1 + hM \right) \norm{f - u}_{\cH_0^1} \left( \norm{f + u}_{\cH_0^1} + \norm{U^{k - 1}}_{\cH_0^1} \right). 
\end{align*}
\end{proposition}

\begin{proof}
Using the definition of the energy functional in \eqref{eq:energy_functional} and that $\cL$ is linear and self-adjoint, we have
\begin{align}\label{eq:con-1}
\big| I^k \left( f \right) - I^k (u) \big| 
    &= \left| \frac{1}{2} \norm{f - U^{k - 1}}^2_{L^2} +   \frac{h}{2}\ip{\cL f,f}_{\cH^{-1},\cH_0^1} + h\ip{F \left( U^{k - 1} \right),f}_{L^2}  \right. \nonumber\\ 
    &\qquad \left. - \frac{1}{2} \norm{u - U^{k - 1}}^2_{L^2} - \frac{h}{2}\ip{\cL u,u}_{\cH^{-1},\cH_0^1} -h\ip{F \left( U^{k - 1} \right),u}_{L^2} \right| \nonumber\\
    &\leq \frac{1}{2} \Big|\underbrace{\norm{f}^2_{L^2} - \norm{u}^2_{L^2}-2\ip{f-u,U^{k-1}}_{L^2}}_{= \ip{f-u,f+u-2U^{k+1}}}\Big| \nonumber\\
    &\quad + \frac{h}{2} \abs{ \ip{\cL (f-u), f+u}_{\cH^{-1},\cH_0^1} }
        + h \abs{\ip{F \left( U^{k - 1} \right),f - u}_{L^2}} \nonumber\\
    &\stackrel{\text{CS}}{\le} \frac{1}{2} \norm{f - u}_{L^2} \norm{f + u - 2 U^{k-1}}_{L^2} + \frac{h}{2} \abs{\ip{\cL (f-u),f+u}_{\cH^{-1},\cH_0^1}}\nonumber\\ 
    &\quad  + h \abs{\ip{F \left( U^{k - 1} \right),f-u}_{L^2}} .
\end{align}
Moreover, using \cref{ass:bound} and the Cauchy--Schwarz inequality again, we get
\begin{align}\label{eq:con-2}
\frac{h}{2} \abs{\ip{\cL (f-u),f+u}_{\cH^{-1},\cH_0^1}} & + h \abs{\ip{F \left( U^{k - 1} \right),f-u}_{L^2}} \nonumber\\ 
& \le \frac{h M}{2}\norm{  f - u }_{\cH_0^1}\norm{ f + u }_{\cH_0^1} + hM\norm{ U^{k - 1}}_{\cH_0^1}\norm{f - u}_{L^2} \nonumber\\ 
& \le  hM \left( \norm{f + u}_{\cH_0^1} + \norm{U^{k - 1}}_{\cH_0^1} \right) \norm{f - u}_{\cH_0^1},
\end{align}
while from the triangle inequality, we have that
\begin{align}\label{eq:con-3}
 \frac{1}{2} \norm{f - u}_{L^2}\norm{f + u - 2U^{k - 1}}_{L^2} \le \norm{f - u}_{\cH_0^1} \left( \norm{f + u}_{\cH_0^1} + \norm{U^{k - 1}}_{\cH_0^1} \right) . 
\end{align}
Replacing \eqref{eq:con-2} and \eqref{eq:con-3} into \eqref{eq:con-1} completes the proof.
\end{proof}

\begin{proof}[Proof of \cref{thm:5}]
Assume that $\norm{w_m - w_*}_{\cH_0^1} \to 0$. 
Then, the sequence $\norm{w_m - w_*}_{\cH_0^1}$ is bounded by some constant $C > 0$. 
Using \cref{prop:functional_convergence}, we get that
\begin{align*}
 \abs{I^k \left( w_m \right) - I^k \left( w_* \right)}\le \norm{w_m - w_*}_{\cH_0^1} C \left( 1 + \norm{U^{k - 1}}_{\cH_0^1} \right) \to 0.
\end{align*}
Thus $I^k \left( w_m \right)\to I^k \left( w_* \right)$. 

Next, we prove that $I^k \left( w_m \right)\to I^k \left( w_* \right)$ implies that $w_m\to w_*$ in $\cH_0^1$. 
Let us first notice that $w_m \rightharpoonup w_*$. 
Otherwise, there exists a subsequence $(w_{m_i})$, an $\e > 0$ and a nonzero functional $f$ such that $\abs{f \left[ w_{m_i} \right] - f \left[ w_* \right] } \ge \e$. 
Since $(w_{m_i})$ is bounded in $\cH_0^1$ (otherwise $I^k \left( w_{m_i} \right)$ is unbounded), it is pre-weakly compact (which means it has a weakly convergent subsequence, see \textit{e.g.} \citet[Corollary 4.56]{Jan}). 
Let us denote one of its weak limits by $w_{**}$. 
Using \cref{lem:weak_con}, $I^k \left( w_* \right) = \lim_{i \to \infty} I^k \left( w_{m_i} \right) \ge I^k \left( w_{**} \right)$. 
This inequality implies $w_* = w_{**}$ by the uniqueness of the minimizer. 
Hence $w_m \rightharpoonup w_*$, a contradiction with $\abs{f \left[ w_{m_i} \right] - f \left[ w_* \right] } \ge \e$. 
Therefore, $w_m \rightharpoonup w_*$.

Now, since $ I^k \left( w_m \right)\to I^k \left( w_* \right)$, $\mathcal{G}^n (w_m)\to \mathcal{G}^n(w_*)$, and $\mathcal{M}^k \left( w_m \right) \to \mathcal{M}^k \left( w_* \right)$, we have that
\begin{align*}
\frac{1}{2} \norm{w_m}^2_{L^2} + \frac{h}{2}\ip{\cL w_m,w_m}_{\cH^{-1},\cH_0^1} 
    \to \frac{1}{2} \norm{w_*}^2_{L^2} + \frac{h}{2} \ip{\cL w_*,w_*}_{\cH^{-1},\cH_0^1}.
\end{align*}
This convergence implies
\begin{align*}
\left( \frac{1 }{2}-\frac{h\lambda_2}{2} \right) \norm{w_m - w_*}^2_{L^2} &+ \frac{h\lambda_2}{2}\norm{w_m - w_*}^2_{\cH_0^1} \\
    &\le \frac{1 }{2} \norm{w_m - w_*}^2_{L^2} + \frac{h}{2} \ip{\cL (w_m-w_*), w_m -w_*}_{\cH^{-1},\cH_0^1} \to 0,
\end{align*}
since $w_m \rightharpoonup w_*$. 
Therefore, by \cref{ass:ellip}, we conclude $\norm{w_m - w_*}_{\cH_0^1} \to 0$.
\end{proof}


\section{Convergence of the training error}
\label{sec:training}

In this section, we show that for each fixed time step $k$, the trained neural network converges to the true solution of the discretized PDE \eqref{eq:discretization} as the number of neurons and the training time tend to infinity.
Therefore, using the convergence of the time-stepping scheme, we can conclude the convergence of the training error.


\subsection{Convergence of the trained neural network}

In this subsection, we analyze the training of the neural network for the deep gradient flow method as a function of the number of neurons $n$.
In particular, we would like to study the training process of the parameters $\theta^n$ as $n\to\infty$, such that the neural network introduced in \cref{def:NN} approximates the solution of the discretized PDE \eqref{eq:discretization}.
We show that this process satisfies a gradient flow equation as the number of neurons tends to infinity, \emph{i.e.} in the so-called ``wide network limit".

Let us denote the parameters of the neural network by $\theta^n = \left( \beta^i, \alpha^i, c^i \right)^n_{i = 1} \in \mathbb{R}^n \times \R^n \times \R^{d \times n}$. 
Moreover, for $\frac{1}{2} < \delta < 1$, let us introduce a neural network
\begin{align*}
    V^n \left( \theta^n; x \right) &= \frac{1}{n^{\delta}} \sum_{i = 1}^n \hat\beta^{i,n} \psi \left( \hat\alpha^{i,n} x + \hat c^{i,n} \right),
\intertext{in accordance with \cref{def:NN}, where the ``clipped" parameters are defined as follows:}
    \hat\alpha^{i, n} & = \begin{cases} \begin{aligned} \left( r_n \land \alpha^i \right) \lor \frac{1}{r_n}, \quad & \text{for } \alpha^i>0, \\
    \left( \frac{-1}{r_n} \land \alpha^i \right) \lor (-r_n), \quad & \text{for } \alpha^i < 0, \end{aligned} \end{cases} \\
    \hat\beta^{i, n} & = \left( r_n \land \beta^i \right) \lor (-r_n), &&\\
    \hat c^{i, n} & = \left( r_n \land c^i \right) \lor (- r_n) , &&
\end{align*}
for some $r_n$ increasing with $n$. We restrict the domain of the parameters $(\beta^i, \alpha^i, c^i)$ to $[-r_n, r_n]$ which converges to $\mathbb{R}$ as $n \to\infty$, and for $\alpha^i$ we also need to subtract the ball $\left( -\frac{1}{r_n}, \frac{1}{r_n} \right)$.
Gradient clipping is in accordance with deep learning literature, see, for example, \citet{Zhang2020Why} and \citet[Ch. 10 and 11]{Goodfellow_2016}.


Next, let us introduce the gradient descent dynamics for the training process of the parameters $\theta^n$, where $t$ denotes the training time. 
The neural network $V^n(\theta^n;\cdot)$ should minimize the loss functional $I^k$ in \eqref{eq:energy_functional} of the deep gradient flow method. 
Hence, the dynamic of $\theta_t^n$ should match the gradient of $I^k ( V^n ; \cdot )$, \emph{i.e.}
\begin{align}
\label{eq:dytheta}
 \frac{\ud \theta_t^n}{\ud t} = & - \eta_n \nabla_\theta I^k \left( V^n \left( \theta_t^n; x \right) \right) \nonumber \\ 
 = & - \eta_n\ip{\D I^k \left( V_t^n \right), \nabla_\theta V_t^n}_{\cH_0^1}, 
 \end{align}
with learning rate $\eta_n = n^{2\delta - 1}$, where $\D$ denotes the Fr\'echet derivative; \emph{i.e.} for any $u, v \in \cH_0^1 (\R^d)$,
\begin{align}
\label{eq:frechet-loss}
 \ip{\D I^k \left( v \right), u}_{\cH_0^1} = & \ip{v-U^{k-1},u}_{L^2} + h \ip{\cL v,u}_{\cH^{-1},\cH_0^1} + h \ip{F \left( U^{k - 1} \right),u}_{L^2}.
\end{align}
We obtain a coordinate dynamic ${(\theta^{n}_t)_{t\ge0}}=(\theta_t^{i, n})_{t\ge0} = \big( \beta_t^{i, n}, \alpha_t^{i, n}, c_t^{i, n}\big)_{t\ge0}$. 
This dynamic depends on the number of hidden layers $n$ of the neural network, since the parameters that optimally approximate a function depend on the number of parameters we use. 
We use a random initialization for this process, independent of $n$, denoted by:
\[
\left( \beta_0^{i, n}, \alpha_0^{i, n}, c_0^{i, n} \right) = \left( \beta_0^i, \alpha_0^i, c_0^i \right) = \theta^i_0.
\]

\begin{namedassumption}{(NNI)} 
\label{ass:initial}
The parameters $\beta_0^i, \alpha_0^i, c_0^i$ that initialize the neural network are i.i.d. random variables that satisfy:
\begin{itemize}
    \item[(i)] $\beta_0^i$ is a symmetric random variable with finite second moment: $\E \left[ \abs{\beta_0^i}^2 \right] < + \infty$;
    \item[(ii)] $\alpha_0^i \ne 0$ $\p$-almost surely and $\E \left[ \abs{\alpha_0^i}^{d+7} + \abs{\alpha_0^i}^{-d-2} \right] < + \infty$;
    \item[(iii)] $c_0^i$ is an $\R^d$-valued random variable and $\E \left[ \abs{c_0^i}^{d+7} \right] < + \infty$;
    \item[(iv)] $\alpha_0^i,\  c_0^i$, have full support, \emph{i.e.} for any Borel set $A\subset \R$ and $B\subset \R^d$ with positive Lebesgue measure, $\p\left(\alpha_0^i\in A\right)$ and $\p\left(c_0^i\in  B\right)$ are positive. 
\end{itemize}
\end{namedassumption}

Using the chain rule, the dynamic $V_t^n (x) = V^n \left( \theta_t^n; x \right)$ satisfies the following equation
\begin{align}
\label{eq:dynamic_VN}
\frac{\ud V_t^n (x)}{\ud t} 
    &= \nabla_\theta V^n \left( \theta_t^n; x \right) \cdot \frac{\ud \theta^n_t}{\ud t} 
    = - \eta_n \nabla_\theta V^n \left( \theta_t^n; x \right) \cdot \nabla_\theta I^k \left( V^n \left( \theta_t^n; x \right) \right) \nonumber \\ 
    &=- \ip{\mathcal{D}I^k \left( V^n_t \right), Z_t^n( x, \cdot)}_{\cH_0^1},
\intertext{with $V_0^n (x) = V^n \left( \theta_0^n; x \right)$ and}
Z_t^n \left( x, y \right) 
    &= \eta_n \nabla_\theta V_t^n (x) \cdot \nabla_\theta V_t^n \left( y \right) \nonumber\\
    &= \frac{1}{n} \sum_{i = 1}^n \nabla _{\theta} \hat\beta^{i,n}_t \psi \left( \hat\alpha^{i, n}_t x + \hat c^{i, n}_t \right) \cdot \nabla _{\theta} \hat\beta^{i,n}_t \psi \left( \hat\alpha^{i, n}_t y + \hat c^{i, n}_t \right). \nonumber
\end{align}
We expect \eqref{eq:dynamic_VN} to converge to the following gradient flow
\begin{align}
\label{eq:whole}
\frac{\ud V_t (x)}{\ud t} &= - \ip{\mathcal{D} I^k \left( V_t \right), Z( x, \cdot)}_{\cH_0^1},
\intertext{with $V_0 = 0$, where}
Z(x,y) &= 
    \E \big[ \nabla _\theta\beta^1_0 \psi \left( \alpha^1_0 x + c^1_0 \right) \cdot \nabla _\theta\beta^1_0 \psi \left( \alpha^1_0 y + c^1_0 \right) \big], \nonumber 
\end{align}
while the inner product of the Fr\'echet derivative of the loss functional with another functional is defined in \eqref{eq:frechet-loss}.
The gradient flow \eqref{eq:whole} is an infinite-dimensional ODE that governs the dynamics of the wide network limit of the neural network during the training process and, the right-hand side depends on the loss function of the gradient flow method for the solution of PDEs in \eqref{eq:energy_functional}.
Let us also point out that the kernel $Z(\cdot,\cdot)$ is not the standard neural tangent kernel, as it also depends on the loss functional, which further complicates the analysis.
The right-hand side of \eqref{eq:whole}, using \eqref{eq:frechet-loss}, takes the following form:
\[
\begin{aligned}
\T\left(v\right)\left(x\right) 
    &:=  \ip{\mathcal{D} I^k \left( v \right), Z( x, \cdot)}_{\cH_0^1} \nonumber\\
    &= \ip{v-U^{k-1},Z( x, \cdot)}_{L^2} + h\ip{\cL v,Z( x, \cdot)}_{\cH^{-1},\cH_0^1}+ h\ip{F \left( U^{k-1} \right),Z( x, \cdot)}_{L^2}.
\end{aligned}
\]
The analysis of this operator plays a crucial role in the next subsection, where we study the long term behavior of this gradient flow.

Let us introduce the following shorthand notation:
\[
    \mathcal X(\theta; x) := \nabla _\theta \beta \psi \left( \alpha x + c \right)
       \quad  \text{and} \quad
    \mathcal X^n(\theta; x ) := \nabla _\theta \hat\beta \psi \left( \hat\alpha x + \hat c \right)
\]
for some generic parameters $\theta=\left(\beta, \alpha, c\right)\in \R \times \R \times \R^d$, where $(\hat\beta, \hat\alpha, \hat c)$ denotes the clipped version of these parameters. 
Moreover, in order to simplify the notation, let us set 
\[
    X (x):= \mathcal X\left(\theta^1_0;x\right), 
        \quad 
    X^n (x) := \mathcal X^n \left( \theta^1_0;x \right)
        \quad \text{and} \quad 
    X_t^{i,n}(x) := \mathcal X^n \left( \theta^{i,n}_t;x \right). 
\]
Then, using this notation we have that
\begin{align*}
Z_t^n ( x, y ) & = \frac{1}{n} \sum_{i = 1}^n X_t^{i,n}\left(x\right) \cdot X_t^{i,n}\left(y\right), \\
\intertext{and}
Z ( x, y) & = \E \left[ X\left(x\right) \cdot X\left(y\right) \right].
\end{align*}
This representation invites us to use the law of large numbers to conclude that $Z^n_t(x,y)\to Z(x,y)$ as $n\to\infty$.
Intuitively, this is the connection between the gradient flow in \eqref{eq:dynamic_VN} that the neural network follows during the training process, and the corresponding ``wide network limit" in \eqref{eq:whole}.

The main result of this subsection follows, which states that as the number of neurons tends to infinity during the training process, then the neural network $V^n$ converges to the wide network limit $V$, which satisfies the gradient flow \eqref{eq:whole}.

\begin{theorem}
\label{thm:train_NN}
Assume that the neural network satisfies \cref{ass:initial}, and let $r_n$ increase with $n$, while $r_n\le \log n$.
Moreover, assume that the operators of the PDE \eqref{eq:general_operator}--\eqref{eq:operator} satisfy \cref{ass:bound,ass:ellip}.
Then, the dynamic \eqref{eq:dynamic_VN} converges to the gradient flow \eqref{eq:whole} as $n \to \infty$, \emph{i.e.} for any $T>0,$
\begin{align*}
    \sup_{0\le t\le T}\E \left[ \big\| V^n_t - V_t \big\|_{\cH_0^1} \right] \xrightarrow[n\to\infty]{} 0. 
\end{align*}
\end{theorem}

\begin{lemma}
\label{lem:v0}
Assume that the neural network satisfies \cref{ass:initial}, then
\begin{align*}
    \E \left[ \norm{V^n_0}_{\cH_0^1} \right] \le C^{\left( 1 \right)}n^{\frac{1}{2} - \delta},
\end{align*}
where $C^{\left( 1 \right)} = \E \left[ \abs{\beta_0^i}^2 \right]^{\frac{1}{2}} \left( \E \left[ \abs{\alpha_0^i}^{-d} \right] + \E \left[ \abs{\alpha_0^i}^{2-d} \right] + 2 \right)^{\frac{1}{2}} \norm{\psi}_{\cH_0^1}$.
\end{lemma}

\begin{proof}
Let us denote a neuron by $Y^i (x) := \hat\beta_0^i \psi \left( \hat\alpha_0^i x + \hat c_0^i \right)$, then $V_0^n = \frac{1}{n^{\delta}} \sum_{i = 1}^n Y^i (x)$ and
\begin{equation}
\label{eq:initial_network_expectation}
\E \left[ \norm{V^n_0}_{\cH_0^1}^2 \right] = \E \left[ \int_{\mathbb{R}^d} \left| V^n_0 \right|^2 + \left| \nabla_x V^n_0 \right|^2 \ud x \right].
\end{equation}
Let us first compute the value of the cross terms, for $i \neq j$, which equals
\begin{align*}
\E \left[ \int_{\mathbb{R}^d} Y^i(x) Y^j(x) \ud x \right] 
    &= \E \left[ \int_{\mathbb{R}^d} \hat\beta_0^i \psi \left( \hat\alpha_0^i x + \hat c_0^i \right) \hat\beta_0^j \psi \left( \hat\alpha_0^j x + \hat c_0^j \right) \ud x \right] \\ 
    &= \int_{\mathbb{R}^d} \E \left[ \hat\beta_0^i \right] \E \left[ \hat\beta_0^j \right] \E \left[ \psi \left( \hat\alpha_0^i x + \hat c_0^i \right) \psi \left( \hat\alpha_0^j x + \hat c_0^j \right) \right] \ud x
    = 0;
\end{align*}
here, we can apply Fubini's theorem since $\psi$ has compact support and the parameters $\theta$ are bounded, while the random variables $\beta^i_0$ are symmetric, hence their expectation is zero.
Therefore, we can bound the $L^2$-norm of $V_0^n$ by
\begin{align}
\label{eq:bound0}
\E \left[ \int_{\mathbb{R}^d} \left| V^n_0 \right|^2 \ud x \right] 
    &= \frac{1}{n^{2 \delta}} \E \left[ \int_{\mathbb{R}^d} \sum_{i = 1}^n \abs{\hat\beta_0^i}^2 \abs{\psi \left( \hat\alpha_0^i x + \hat c_0^i \right)}^2 \ud x \right] \nonumber\\ 
    &= \frac{1}{n^{2 \delta-1}} \E \left[ \abs{\hat\beta_0^i}^2 \int_{\mathbb{R}^d} \abs{\hat\alpha_0^i}^{-d} \abs{\psi (y)}^2 \ud y \right] \nonumber\\ 
    &\leq \frac{1}{n^{2 \delta-1}} \E \left[ \abs{\hat\beta_0^i}^2 \right] \E \left[ \abs{\hat\alpha_0^i}^{-d} \right] \int_{\mathbb{R}^d} \abs{\psi(y)}^2 \ud y.
\end{align}
The derivative of $V^n_0$ can be expressed as
\[
 \nabla_x V^n_0 
    = \frac{1}{n^{\delta}} \sum_{i = 1}^n \nabla_x Y^i (x) 
    = \frac{1}{n^{\delta}} \sum_{i = 1}^n \hat\beta_0^i \hat\alpha_0^i \left( \nabla \psi \right) \left( \hat\alpha_0^i x + \hat c_0^i \right).
\]
Hence, we can analogously bound the $L^2$-norm of the derivative of $V^n_0$ by
\begin{align}
\label{eq:bound1}
\E \left[ \int_{\mathbb{R}^d} \left| \nabla_x V^n_0 \right|^2 \ud x \right] 
    &= \frac{1}{n^{2 \delta-1}}  \E \left[ \int_{\mathbb{R}^d} \abs{\hat\beta_0^i}^2 \abs{\hat\alpha_0^i}^{2} \abs{\left( \nabla \psi \right) \left( \hat\alpha_0^i x + \hat c_0^i \right)}^2 \ud x \right] \nonumber\\ 
    &\leq \frac{1}{n^{2 \delta-1}} \E \left[ \abs{\hat\beta_0^i}^2 \right] \E \left[ \abs{\hat\alpha_0^i}^{2-d} \right] \int_{\mathbb{R}^d} \abs{\nabla \psi (x)}^2 \ud x.
\end{align}
Applying bounds \eqref{eq:bound0} and \eqref{eq:bound1} to \eqref{eq:initial_network_expectation}, and using Jensen's inequality and that $\abs{\hat\beta_0^i} \leq \abs{\beta_0^i}$, yields
\begin{align*}
\E \left[ \norm{V^n_0}_{\cH_0^1} \right] 
    &\leq \sqrt{\E \left[ \norm{V^n_0}_{\cH_0^1}^2 \right] } \\ 
    &\leq \sqrt{ \frac{1}{n^{2 \delta-1}} \E \left[ \abs{\beta_0^i}^2 \right] \E \left[ \abs{\hat \alpha_0^i}^{-d} \right] \int_{\mathbb{R}^d} \abs{\psi}^2 \ud x + \frac{1}{n^{2 \delta-1}} \E \left[ \abs{\beta_0^i}^2 \right] \E \left[ \abs{\hat \alpha_0^i}^{2-d} \right] \int_{\mathbb{R}^d} \abs{\nabla \psi}^2 \ud x} \\ 
    &\leq n^{\frac12 - \delta} \E \left[ \abs{\beta_0^i}^2 \right]^{\frac{1}{2}} \left( \E \left[ \abs{\hat \alpha_0^i}^{-d} \right] + \E \left[ \abs{\hat \alpha_0^i}^{2-d} \right] \right)^{\frac{1}{2}} \sqrt{\int_{\mathbb{R}^d} \abs{\psi}^2 \ud x + \int_{\mathbb{R}^d} \abs{\nabla \psi}^2 \ud x}.
\end{align*}
If $\abs{\alpha} \leq r_n$, then $\abs{\hat \alpha}^{-1} \leq \abs{\alpha}^{-1}$ and if $\abs{\alpha} > r_n$, then $\abs{\hat \alpha}^{-1} = r_n^{-1}$. 
Therefore, for $d=1,2$, we have
\begin{align*}
\E \left[ \abs{\hat \alpha_0^i}^{-d} \right] + \E \left[ \abs{\hat \alpha_0^i}^{2-d} \right]
    &\le \E \left[ \abs{\alpha_0^1}^{-d} + \left(r_n\right)^{-d} + \abs{\alpha_0^1}^{2-d} \right],
\intertext{while for $d \ge 3$, we get}
\E \left[ \abs{\hat \alpha_0^i}^{-d} \right] + \E \left[ \abs{\hat \alpha_0^i}^{2-d} \right] 
    &\le   \E \left[ \abs{\alpha_0^1}^{-d} + \left(r_n\right)^{-d} + \abs{\alpha_0^1}^{2-d} + \left(r_n\right)^{2-d} \right]. 
\qedhere
\end{align*}
\end{proof}

\begin{proof}[Proof of \cref{thm:train_NN}]
Using \eqref{eq:dynamic_VN} and \eqref{eq:whole}, we need to estimate the following difference:
\begin{align*}
V^n_t (x) - V_t (x) 
    &= V^n_0 (x) - V_0 (x) + \int_0^t \Big\{ \ip{\mathcal{D} I^k \left( V_s \right), Z( x, \cdot)}_{\cH_0^1} - \ip{\mathcal{D}I^k \left( V^n_s \right), Z_s^n \left( x, \cdot\right)}_{\cH_0^1} \Big\} \, \ud s \\
    &= \underbrace{V^n_0 (x) - V_0 (x)}_{\left( a1 \right)} + \int_0^t \Big\{ \underbrace{\ip{\mathcal{D}I^k \left( V_s \right) - \mathcal{D}I^k \left( V^n_s \right), Z( x, \cdot)}_{\cH_0^1}}_{\left( a2 \right)} \\
    &\qquad+ \underbrace{\ip{\mathcal{D}I^k \left( V^n_s \right), Z( x, \cdot) - Z_0^n \left( x, \cdot\right)}_{\cH_0^1}}_{\left( a3 \right)} + \underbrace{\ip{\mathcal{D}I^k \left( V^n_s \right), Z^n_0 \left( x, \cdot\right) - Z_s^n \left( x, \cdot\right)}_{\cH_0^1}}_{\left( a4 \right)} \Big\} \, \ud s.
\end{align*}
The proof is now separated in several steps, corresponding to the estimation of each of the terms above.

\noindent\textit{Step (a1).}
We know that $V_0(x)=0$ by definition, hence, using \cref{lem:v0}, we get 
\[
    \E \left[ \norm{V^n_0 (x) - V_0 (x)}_{\cH_0^1} \right] \le C^{\left( 1 \right)}n^{\frac{1}{2} - \delta}. 
\]

\noindent\textit{Step (a2).}
In order to separate the random terms $V$ and $X$, we define another probability space $\Omega'$ and probability measure $\p'$ such that $X\left(\omega'\right)$ under $\p'$ has the same distribution as $X$ under $\p.$ 
Then, we can rewrite $(a2)$ as follows
\begin{align*}
(a2)
    &= \ip{\mathcal{D}I^k \left( V_s \right) - \mathcal{D}I^k \left( V^n_s \right), \E\left[X\cdot X\left(x\right)\right]}_{\cH_0^1} \\
    &= \int_{\Omega'}\ip{\mathcal{D}I^k \left( V_s \right) - \mathcal{D}I^k \left( V^n_s \right), X\left(\omega'\right)}_{\cH_0^1} X \left(x\right)\left(\omega'\right) \p' \left( \ud \omega' \right). 
\end{align*}
Hence, we can bound the $\cH_0^1$-norm of (a2) by
\begin{align*}
\norm{(a2)}_{\cH_0^1} 
    &\le \int_{\Omega'} \norm{\ip{\mathcal{D}I^k \left( V_s \right) - \mathcal{D}I^k \left( V^n_s \right), X\left(\omega'\right)}_{\cH_0^1} \cdot X\left(x\right)\left(\omega'\right)}_{\cH_0^1} \p \left( \ud \omega' \right) \\
    &= \int_{\Omega'}\abs{\ip{\mathcal{D}I^k \left( V_s \right) - \mathcal{D}I^k \left( V^n_s \right), X\left(\omega'\right)}_{\cH_0^1}} \norm{ X\left(\omega'\right)}_{\cH_0^1} \p \left( \ud \omega' \right) \\
    &\hspace{-1.35em}\overset{\text{\cref{lem:con_frechet}}}{\le} K \norm{V_s - V_s^n}_{\cH_0^1} \int_{\Omega'} \norm{ X\left(\omega'\right)}_{\cH_0^1}^2 \p \left( \ud \omega' \right) \\
    &\hspace{-1.35em}\overset{\text{\cref{lem:bound_XH}}}{=} K \norm{V_s - V_s^n}_{\cH_0^1} \E\left[\norm{ X}_{\cH_0^1}^2\right] \\
    &\le C_\psi \norm{V_s - V_s^n}_{\cH_0^1}.
\end{align*}
Here, $K$ is the constant from \cref{lem:con_frechet} and $C_\psi$ is another constant that depends on the activation function $\psi$ and may change from line to line.



\noindent\textit {Step (a3).}
Let us rewrite the term $\left(a3\right)$ as follows:
\begin{align*}
\left(a3\right)
    &= \underbrace{\frac{1}{n}\sum_{i=1}^n \ip{\mathcal{D}I^k \left( V^n_s \right), X^{i,n}_0\cdot X^{i,n}_0\left(x\right)-\E\left[X^{i,n}_0\cdot X^{i,n}_0\left(x\right)\right]}_{\cH_0^1}}_{\left(a3.1\right)}\\
    &\quad+ \underbrace{\ip{\mathcal{D}I^k \left( V^n_s \right), \E\left[X^{n}\cdot X^{n}\left(x\right)-X\cdot X\left(x\right)\right]}_{\cH_0^1}}_{\left(a3.2\right)},
\end{align*}
where we can use $X^n $ instead of $X^{i,n}_0$ in the second term since, by \cref{ass:initial}, these two terms have the same expectation.
Then, we can further separate the term $\left(a3.2\right)$ as follows:
\begin{align*}
\left(a3.2\right)
    = \underbrace{\ip{\mathcal{D}I^k \left( V^n_s \right), \E\left[\left(X^n-X\right)\cdot X^n\left(x\right)\right]}_{\cH_0^1}}_{\left(a3.21\right)}
    +\underbrace{\ip{\mathcal{D}I^k \left( V^n_s \right), \E\left[X \cdot \left( X^n\left(x\right)-X\left(x\right) \right)\right]}_{\cH_0^1}}_{\left(a3.22\right)}.
\end{align*}
A similar computation as for the term $(a2)$, yields
\begin{align*}
\norm{\left(a3.21\right)}_{\cH_0^1}
    &\le \int_{\Omega'}\abs{\ip{ \mathcal{D}I^k \left( V^n_s \right), X^n\left(\omega'\right)-X\left(\omega'\right)}_{\cH_0^1}}\norm{ X^n\left(\omega'\right)}_{\cH_0^1} \p \left( \ud \omega' \right) \\
    &\hspace{-1.35em}\overset{\text{\cref{lem:con_frechet}}}{\le} K \left( 1 + \norm{V_s^n}_{\cH_0^1} \right) \int_{\Omega'} \norm{ X^n \left( \omega' \right) -X \left( \omega' \right)}_{\cH_0^1} \norm{ X^n \left( \omega' \right)}_{\cH_0^1} \p \left( \ud \omega' \right) \\
    &= K \left( 1 + \norm{V_s^n}_{\cH_0^1} \right) \E \left[ \norm{ X^n - X}_{\cH_0^1} \norm{X^n}_{\cH_0^1} \right]. 
\end{align*}
Analogously, we have that
\[
    \norm{\left(a3.22\right)}_{\cH_0^1} \le K \left( 1 + \norm{V_s^n}_{\cH_0^1} \right) \E \left[ \norm{X^n - X}_{\cH_0^1} \norm{X}_{\cH_0^1} \right].
\]
Overall, combining the two bounds and then using \cref{lem:EVtN}, then the Cauchy--Schwarz inequality, and finally \cref{lem:bound_XH,lem:epN}, we arrive at
\begin{align*}
\E\left[\norm{\left(a3.2\right)}_{\cH_0^1}\right] 
    &\le 2K \E \left[ 1 + \norm{V_s^n}_{\cH_0^1} \right] \E \left[ \norm{X^n - X}_{\cH_0^1} \left(\norm{ X^n}_{\cH_0^1} + \norm{X}_{\cH_0^1} \right) \right] \\
    &\le C_\psi \E \left[  \norm{X^n - X}_{\cH_0^1} \left( \norm{ X^n}_{\cH_0^1} + \norm{X}_{\cH_0^1} \right) \right] \\
    &\le  C_\psi \E \left[ \norm{ X^n-X}_{\cH_0^1}^2 \right]^{\frac{1}{2}} \E \left[ \norm{ X^n}^2_{\cH_0^1} + \norm{ X}^2_{\cH_0^1} \right]^{\frac{1}{2}} 
     \le C_\psi \, \e_n^{\frac{1}{2}},
\end{align*}
where $\e_n$ is defined in \cref{lem:epN} and equals
\[
    \e_n = \E \left[ \norm{ X^n-X}_{\cH_0^1}^2 \right].
\]


On the other hand, using \cref{lem:con_frechet} the norm of $\left(a3.1\right)$ can be bounded by
\begin{align*}
\norm{\left(a3.1\right)}_{L^2} 
    &= \frac{1}{n} \left( \int_{\R^d} \ip{\mathcal{D} I^k \left( V^n_s \right), \sum_{i=1}^n \left( X^{i,n}_0 \cdot X^{i,n}_0 \left(x\right) - \E \left[ X^{i,n}_0 \cdot X^{i,n}_0\left( x \right) \right] \right)}^2_{\cH_0^1} \ud x \right)^{\frac{1}{2}} \\
    &\le \frac{K}{n} \left( 1 + \norm{V^n_s}_{\cH_0^1} \right) \left( \int_{\R^d} \norm{\sum_{i=1}^n \left( X^{i,n}_0 \cdot X^{i,n}_0 \left(x\right) - \E \left[ X^{i,n}_0 \cdot X^{i,n}_0\left(x\right) \right] \right)}^2_{\cH_0^1} \ud x \right)^{\frac{1}{2}}.
\end{align*}
Then, using the Cauchy--Schwarz inequality and \cref{lem:EVtN}, we get that
\begin{align*}
\E \left[ \norm{\left(a3.1\right)}_{L^2} \right]
    &\le \frac{K}{n} \E \left[ \left( 1+ \norm{V^n_s}_{\cH_0^1} \right)^2 \right]^{\frac{1}{2}} \E \left[ \int_{\R^d} \norm{\sum_{i=1}^n \left( X^{i,n}_0 \cdot X^{i,n}_0\left(x\right) - \E \left[ X^{i,n}_0 \cdot X^{i,n}_0\left(x\right) \right] \right)}^2_{\cH_0^1} \ud x \right]^{\frac{1}{2}} \\
    &\leq \frac{C_\psi}{n} \E \left[ \int_{\R^d} \norm{\sum_{i=1}^n \left( X^{i,n}_0 \cdot X^{i,n}_0\left(x\right) - \E \left[ X^{i,n}_0 \cdot X^{i,n}_0\left(x\right) \right] \right)}^2_{\cH_0^1} \ud x \right]^{\frac{1}{2}} \\
    &= \frac{C_\psi}{\sqrt{n}} \E \left[ \int_{\R^d} \norm{ X^n \cdot X^n\left(x\right) - \E \left[ X^n \cdot X^n\left(x\right) \right]}^2_{\cH_0^1} \ud x \right]^{\frac{1}{2}},
\end{align*}
where the last equality follows because $X^n$ and $X^{i,n}_0$ are equally distributed by \cref{ass:initial}, while $X^{i,n}\cdot X^{i,n}\left(x\right)$ are i.i.d. variables that satisfy
\[
\E \left[ \ip{X^{i,n} \cdot X^{i,n}\left(x\right), X^{j,n} \cdot X^{j,n}\left(x\right)}_{\cH_0^1} \right] = 0, \quad \text{for } i \ne j.
\] 
Using the triangle inequality, we have
\begin{align*}
&\E \left[ \int_{\R^d} \norm{X^n \cdot X^n\left(x\right) - \E \left[ X^n \cdot X^n\left(x\right) \right]}^2_{\cH_0^1} \ud x \right] \\
    &\qquad\qquad \le 2 \E \left[ \int_{\R^d} \norm{X^n \cdot X^n\left(x\right)}^2_{\cH_0^1} \ud x \right] + 2\int_{\R^d} \norm{ \E \left[ X^n \cdot X^n\left(x\right) \right]}^2_{\cH_0^1} \ud x \\
    &\qquad\qquad \le 2 \E \left[ \norm{X^n}^2_{\cH_0^1} \norm{X^n}^2_{L^2} \right] + 2 \E \left[ \norm{X^n}^2_{\cH_0^1} \right] \E \left[ \norm{X^n}_{L^2}^2 \right].
\end{align*}
Then, combining the last two inequalities, we arrive at
\begin{align*}
\E\left[\norm{\left(a3.1\right)}_{L^2} \right] 
    \le \frac{C_\psi}{\sqrt{n}} \left( \E \left[ \norm{X^n}^2_{\cH_0^1} \norm{X^n}^2_{L^2} \right]^{\frac{1}{2}} + \E \left[ \norm{X^n}^2_{\cH_0^1} \right]^{\frac{1}{2}} \E \left[ \norm{X^n}^2_{L^2} \right]^{\frac{1}{2}} \right).
\end{align*}
We can analogously estimate the term $\E\left[\norm{\nabla_x\left(a3.1\right)}_{L^2} \right]$ and deduce that 
\begin{align*}
\E\left[\norm{\left(a3.1\right)}_{\cH_0^1}\right] 
    &\le \frac{C_\psi}{\sqrt{n}}\left( \E \left[ \norm{X^n}^4_{\cH_0^1} \right]^{\frac{1}{2}} + \E \left[ \norm{X^n}^2_{\cH_0^1} \right] \right) 
     \le \frac{C_\psi}{\sqrt{n}} \E \left[ \norm{X^n}^4_{\cH_0^1} \right]^{\frac{1}{2}} \\
    &\hspace{-1.35em}\overset{\text{\cref{lem:boundXN}}}{\le} \frac{C_\psi\left(r_n\right)^{8+d}}{\sqrt{n}}.
\end{align*}
Overall, we finish this step by concluding that 
\[
    \E\left[ \norm{\left(a3\right)}_{\cH_0^1} \right]\le C_\psi\left(  \frac{\left(r_n\right)^{8+d}}{\sqrt{n}}+\e_n^{\frac{1}{2}}\right).
\]


\noindent \textit{Step (a4).}
Recalling the definition of $Z_t^n,$ we have
\begin{align*}
\left(a4\right) 
    &= \frac{1}{n}\sum_{i=1}^n \ip{\mathcal{D}I^k \left( V^n_s \right), X^{i,n}_0 }_{\cH_0^1}\cdot X^{i,n}_0 (x) - \frac{1}{n}\sum_{i=1}^n \ip{\mathcal{D}I^k \left( V^n_s \right), X_s^{i,n} }_{\cH_0^1}\cdot X_s^{i,n} (x) \\
    &= \frac{1}{n}\sum_{i=1}^n \ip{\mathcal{D}I^k \left( V^n_s \right), X^{i,n}_0 - X_s^{i,n}}_{\cH_0^1}\cdot X^{i,n}_0 \left(x\right) \\
    &\qquad - \frac{1}{n} \sum_{i=1}^n \ip{\mathcal{D}I^k \left( V^n_s \right), X_s^{i,n}}_{\cH_0^1} \cdot \left(X_s^{i,n}\left(x\right) - X^{i,n}_0 \left(x\right) \right).
\end{align*}
Using first \cref{lem:con_frechet} and then \cref{lem:boundXN,lem:lip_Xtheta}, we get
\begin{align*}
\norm{\left(a4\right)}_{\cH_0^1} 
    &\le \frac{K}{n} \left( 1 + \norm{V_s^n}_{\cH_0^1} \right)\sum_{i=1}^n \left( \norm{X_s^{i,n}}_{\cH_0^1} + \norm{X^{i,n}}_{\cH_0^1} \right) \norm{X_s^{i,n} - X^{i,n}}_{\cH_0^1} \\
   &\le \frac{C_\psi}{n} \left(r_n\right)^{8 + d} \left( 1 + \norm{V_s^n}_{\cH_0^1} \right) \sum_{i=1}^n \abs{\theta_s^{i,n} - \theta_0^i}^{\frac{1}{2}}.
\end{align*}
Hence we can bound its expectation by
\begin{align*}
\E \left[ \norm{\left(a4\right)}_{\cH_0^1} \right] 
    &\le \frac{C_\psi}{n} \left(r_n\right)^{8 + d} \E \left[ \left( \norm{V_s^n}_{\cH_0^1}+1 \right)^2 \right]^{\frac{1}{2}} \E \left[ \left( \sum_{i=1}^n \abs{\theta_s^{i,n} - \theta_0^i}^{\frac{1}{2}} \right)^2 \right]^{\frac{1}{2}} \\
    &\le C_\psi n^{-\frac{1}{2}}\left(r_n\right)^{8 + d} \E \left[ \left( \norm{V_s^n}_{\cH_0^1}+1 \right)^2 \right]^{\frac{1}{2}} \E \left[ \abs{\theta_s^{1,n} - \theta_0^1} \right]^{\frac{1}{2}} \\
    &\hspace{-1.95em}\overset{\text{Lem. \ref{lem:EVtN} \& \ref{lem:thetat0}}}{\le}  C_\psi \sqrt{t}n^{\frac{\delta}{2}-1}\left(r_n\right)^{10 + 2d} .
\end{align*}

\noindent\textit{Final step.}
Combining the previous steps, we arrive at 
\begin{align*}
\E \left[ \norm{ V^n_t- V_t}_{\cH_0^1} \right] 
    &\le \E \left[ \norm{\left(a1\right)}_{\cH_0^1} \right] + \E \left[ \norm{ \int_0^t (a2) + \left(a3\right) + \left(a4\right) \ud s}_{\cH_0^1} \right] \\
    &\le \E \left[ \norm{\left(a1\right)}_{\cH_0^1} \right] +  \int_0^t \E \left[\norm{ (a2)}_{\cH_0^1}\right] + \E \left[\norm{ \left(a3\right)}_{\cH_0^1}\right] + \E \left[\norm{ \left(a4\right)}_{\cH_0^1}\right] \ud s  \\
    &\le M \int_0^t \E \left[ \norm{ V^n_s- V_s}_{\cH_0^1} \right] \ud s \\
    &\qquad + C^{\left(1\right)} n^{\frac{1}{2} - \delta} + T^{\frac{3}{2}} C_\psi \left( \frac{\left(r_n\right)^{8+d}}{\sqrt{n}} + \e_n^{\frac{1}{2}} + n^{\frac{\delta}{2}-1}\left(r_n\right)^{10 + 2d} \right).
\end{align*}
Hence, using Gr\"onwall's inequality, we conclude
\begin{align*}
\E \left[ \norm{ V^n_t- V_t}_{\cH_0^1} \right] 
    &\le \ue^{M T} \left( C^{\left(1\right)} n^{\frac{1}{2} - \delta} + T^{\frac{3}{2}} C_\psi \left( \frac{\left(r_n\right)^{8+d}}{\sqrt{n}} + \e_n^{\frac{1}{2}} + n^{\frac{\delta}{2}-1}\left(r_n\right)^{10 + 2d} \right) \right) \\
    & \xrightarrow[n \to \infty]{} 0. \qedhere  
\end{align*}
\end{proof}


\subsection{Long time behavior of the gradient flow}

In this subsection, we prove that the wide network limit of the trained neural network, \textit{i.e.} the process $V_t$ defined in \eqref{eq:whole}, converges to the global minimizer $w_*$ of the loss function $I^k$ of the DGFMs, as the training time $t\to\infty$.
This result, combined with the convergence of the time-stepping scheme, then proves the convergence of the training error.

\begin{theorem}
\label{thm:whole}
Assume that the neural network satisfies \cref{ass:initial} and the coefficients of the PDE \eqref{eq:general_operator}--\eqref{eq:operator} satisfy \cref{ass:bound,ass:ellip}.
Then, we have 
    \begin{align*}
        \lim_{t \to \infty} \norm{V_t - w_*}_{\cH_0^1} = 0.
    \end{align*}
\end{theorem}

Let us start by rewriting the dynamics of the gradient flow $V$ in \eqref{eq:whole} as follows:
\begin{align}
\label{eq:whole_new}
\frac{\ud \left(V_t - w_*\right) (x)}{\ud t} 
    = \frac{\ud \left(V_t \right) (x)}{\ud t} 
    = - \ip{\mathcal{D} I^k \left( V_t - w_* + w_* \right), Z( x, \cdot)}_{\cH_0^1} 
    = -\widetilde \T\left(V_t - w_*\right)\left(x\right),
\end{align}
where
\begin{align*}
\widetilde \T\left(v\right) 
    &:= \T\left(v + w_*\right) \\
    &= \ip{v+w_*-U^{k-1},Z(x,\cdot)}_{L^2} + h \ip{\cL (v+w_*),Z(x,\cdot)}_{\cH^{-1},\cH_0^1} + h \ip{F \left( U^{k - 1} \right),Z(x,\cdot)}_{L^2} \\
    &= \ip{v,Z(x,\cdot)}_{L^2} + h \ip{\cL v,Z(x,\cdot)}_{\cH^{-1},\cH_0^1} + \ip{w_* - U^{k-1} + h \left(\cL w_* + F \left( U^{k - 1} \right) \right) ,Z(x,\cdot)}_{L^2} \\
    &= \ip{v,Z(x,\cdot)}_{L^2} + h \ip{\cL v,Z(x,\cdot)}_{\cH^{-1},\cH_0^1}.
\end{align*}
We work with $\widetilde\T$ in the sequel, because $\T$ is not linear $\left( \T\left(0\right)\ne 0 \right)$.
Next, let us define another inner product, such that $\widetilde \T$ becomes positive semi-definite. 
Indeed, for any $u, v \in \cH_0^1\left(\R^d\right),$ set
\[
    \ip{v,u}_{\widetilde \cH_0^1} :=  \ip{v,u}_{L^2} + h \ip{\cL v,u}_{\cH^{-1},\cH_0^1},
\]
then, using \cref{ass:bound,ass:ellip}, we have
\[
\norm{u}_{\widetilde \cH_0^1}^2 
    = \ip{u,u}_{\widetilde \cH_0^1} 
    = \ip{u,u}_{L^2} + h \ip{\cL u,u}_{\cH^{-1},\cH_0^1} 
    \begin{cases} 
        \leq (1 + h M) \norm{u}_{\cH_0^1}^2 \\
        \geq h \lambda_1 \norm{u}_{\cH_0^1}^2 + (1 - h \lambda_2) \norm{u}_{L^2}^2 \geq h \lambda_1 \norm{u}_{\cH_0^1}^2.
    \end{cases}
\]
Hence, this inner product induces a norm on $\cH_0^1\left(\R^d\right)$, denoted by $\norm{\cdot}_{\widetilde \cH_0^1}$, which is equivalent to $\norm{\cdot}_{\cH_0^1}$. 
In this case, we can rewrite $\widetilde \T\left(v\right)\left(x\right)= \ip{v,Z(x,\cdot)}_{\widetilde \cH_0^1}$.

\begin{proposition}
Assume that the neural network satisfies \cref{ass:initial} and the coefficients of the PDE \eqref{eq:general_operator}--\eqref{eq:operator} satisfy \cref{ass:bound}.
Then, $\widetilde\T$ is a self-adjoint, positive definite, and trace class operator on $\cH^1_0\left(\R^d\right)$ with inner product $\ip{\cdot,\cdot}_{\widetilde \cH_0^1}$ \emph{i.e.}, for any $u,v \in \cH_0^1\left(\R^d\right)$ holds
\begin{align*}
\ip{\widetilde \T \left(v\right), u}_{\widetilde\cH^1_0}= \ip{ v,\widetilde \T (u)}_{\widetilde\cH^1_0},
    \quad \ip{\widetilde \T \left(v\right), v}_{\widetilde\cH^1_0} > 0 \ \text{ for } v \ne 0, 
    \quad \text{ and }  \quad
    \sum_{i=1}^\infty \ip{\widetilde \T \left(\ue_i\right), \ue_i}_{\widetilde\cH^1_0} < +\infty,
\end{align*}
where $\{ \ue_i \}^\infty_{i=1}$ is an orthogonal basis on $\cH_0^1 (\R^d)$ under the norm $\ip{\cdot,\cdot}_{\widetilde \cH_0^1}$.
\end{proposition}

\begin{proof}
Let us first verify that $\widetilde{\mathcal{T}}$ is self-adjoint and positive definite. 
Using that 
\[
    \widetilde \T \left(v\right)\left(x\right) 
        = \ip{v, Z(x,\cdot)}_{\widetilde \cH_0^1} 
        = \E \left[ \ip{v,X}_{\widetilde \cH_0^1} \cdot X \right],
\]
taking the inner product yields
\begin{align*}
\ip{\widetilde \T \left(v\right), u}_{\tilde \cH^1_0} 
    &= \E \left[ \ip{v, X}_{\widetilde \cH_0^1}\cdot \ip{u, X}_{\widetilde \cH_0^1} \right]
     = \ip{ v,\widetilde \T (u)}_{\widetilde\cH^1_0}, 
\intertext{ and}
\ip{\widetilde \T \left(v\right), v}_{\tilde \cH^1_0} 
    &= \E \left[ \abs{\ip{v, X}_{\widetilde \cH_0^1}}^2 \right] \ge 0.
\end{align*}
Next, let us verify that $\E \left[ \abs{\ip{v,X}_{\widetilde \cH_0^1}}^2 \right] = 0$ only if $v = 0$. 
The first marginal of $X$ is $\psi \left( \alpha_0^1 x + c_0^1 \right)$ and we know from \cref{ass:initial} that the random variables $\alpha_0^1$ and $c_0^1$ have full support. 
Hence, $\E \left[ \abs{\ip{v,X}_{\widetilde \cH_0^1}}^2 \right] = 0$ implies $\ip{v,w}_{\widetilde \cH_0^1}=0$ for any $w\in \mathcal{C}\left(\psi\right)$. 
Using that $\mathcal{C}\left(\psi\right)$ is dense in $\cH_0^1\left(\R^d\right)$ with the norm $\norm{\cdot}_{\cH_0^1}$, see \cref{thm:neural app}, then it is dense with the norm $\norm{\cdot}_{\widetilde\cH_0^1}$ as well. 
Therefore, $v=0.$ 

Finally, let us show that $\widetilde{\mathcal{T}}$ is a trace class operator. 
Using Parseval's identity and \cref{lem:bound_XH}, we have
\[
\sum_{i=1}^\infty \ip{\widetilde\T\left(\ue_i\right), \ue_i}_{\widetilde\cH^1_0} 
    = \sum_{i=1}^\infty \E \left[ \abs{\ip{\ue_i, X}_{\widetilde \cH_0^1}}^2 \right] 
    = \E \left[ \norm{X}_{\widetilde \cH_0^1}^2 \right] 
    \le C \E \left[ \norm{X}_{\cH_0^1}^2 \right] 
    < + \infty. \qedhere
\]
\end{proof}

\begin{proof}[Proof of \cref{thm:whole}]
Using \citet[Proposition 14.13]{Jan}, every trace class operator is compact and positive definite. 
Therefore, we can do a spectral decomposition for the operator $\widetilde \T$. 
There exists an orthogonal basis $\{ \tilde \ue_i \}^\infty_{i=1}$, such that 
\begin{align*}
    \widetilde \T \left( \tilde \ue_i \right) = \gamma_i \tilde \ue_i,
\end{align*}
with $\gamma_1 \ge \gamma_2 \ge \dots > 0.$ 
Set $h_t^i := \ip{V_t - w_*, \tilde \ue_i}_{\widetilde \cH_0^1}$. 
Then, using \eqref{eq:whole_new}, we have
\begin{align*}
\frac{\ud h_t^i}{\ud t} 
    = \frac{\ip{\ud \left(V_t - w_*\right), \tilde \ue_i}_{\widetilde \cH_0^1}}{\ud t} 
    = -\ip{\widetilde \T \left(V_t - w_*\right), \tilde \ue_i}_{\widetilde \cH_0^1} 
    = -\ip{V_t - w_*, \widetilde \T \tilde \ue_i}_{\widetilde \cH_0^1} 
    = -\gamma_i h_t^i.
\end{align*}
Therefore, $h_t^i = \ue^{-\gamma_i t} h_0^i$. 
Hence, using Parseval's identity again, we get
\begin{align*}
    \norm{V_t - w_*}^2_{\widetilde \cH_0^1} 
        = \sum_{i=1}^\infty \left( h^i_t \right)^2 
        = \sum_{i=1}^\infty \ue^{-2\gamma_i t} \left( h_0^i \right)^2,
\end{align*}
which converges to $0$ because $\gamma_i > 0$ and $\sum_{i=1}^\infty \left(h_0^i\right)^2 = \norm{w_*}^2_{\widetilde \cH_0^1} < + \infty.$ 
Finally, since the norm $\norm{\cdot}_{\widetilde \cH_0^1}$ is equivalent to $\norm{\cdot}_{ \cH_0^1},$ we conclude
\[
    \lim_{t\to\infty}\norm{V_t-w_*}_{\cH_0^1} = 0. \qedhere
\]
\end{proof}

%


\appendix
\section{Auxiliary results}


\subsection{Functional inequalities and norm estimates}

In the first part of the appendix, we show that the Fréchet derivative of the loss function is continuous, and we also prove that the neural network and its wide network limit are bounded in the $\mathcal H_0^1$-norm.

\begin{lemma}[Continuity of the Fr\'echet derivative] 
\label{lem:con_frechet}
Assume that the operators of the PDE \eqref{eq:general_operator}--\eqref{eq:operator} satisfy \cref{ass:bound}. 
Then, the Fr\'echet derivative of the loss function is continuous, \emph{i.e.} there exists a constant $K > 0$, such that for any $u, v, w$ in $\cH_0^1 (\R^d)$ holds
\begin{align*}
    \abs{\ip{\D I^k \left( v \right), u}_{\cH_0^1} - \ip{\D I^k \left( w \right), u}_{\cH_0^1}} \le K  \norm{v - w}_{\cH_0^1 }  \norm{u}_{\cH_0^1 } .
\end{align*}
In particular, by choosing $w = 0$, we  have
\begin{align*}
    \abs{\ip{\D I^k \left( v \right), u}_{\cH_0^1}} \le K \left(  1 + \norm{v}_{\cH_0^1 } \right)  \norm{u}_{\cH_0^1 } .
\end{align*}
\end{lemma}

\begin{proof}
Using the definition of the Fr\'echet derivative in \eqref{eq:frechet-loss}, the triangle and Cauchy--Schwarz inequalities and \cref{ass:bound}, we get that
\begin{align*}
\abs{\ip{\D I^k \left( v \right), u}_{\cH_0^1} - \ip{\D I^k \left( w \right), u}_{\cH_0^1}} 
    &=  \abs{\ip{v-w,u}_{L^2} + \frac{h}{2}\ip{\cL (v-w),u}_{\cH^{-1},\cH_0^1}} \\  
    &\leq \norm{v-w}_{L^2} \norm{u}_{L^2}+ \frac{hM}{2}\norm{v-w}_{\cH_0^1}\norm{u}_{\cH_0^1} \\ 
    &\le K  \norm{v - w}_{\cH_0^1 }  \norm{u}_{\cH_0^1 }.
\end{align*} 
Setting $w=0$ and using again the Cauchy--Schwarz inequality and \cref{ass:bound}, yields
\begin{align*}
\abs{\ip{\D I^k \left( 0 \right), u}_{\cH_0^1}} 
    & = \left| -\ip{U^{k-1},u}_{L^2} +h\ip{F(U^{k-1}),u}_{L^2} \right| \\
    & \le \norm{U^{k-1}}_{L^2}\norm{u}_{L^2}+h\norm{F(U^{k-1})}_{L^2} \norm{u}_{L^2} 
    \leq K \norm{u}_{\cH_0^1 },
\end{align*}
since $U^{k-1} \in \cH_0^1$. 
Then, combining the two results and using the triangle inequality, we arrive at
\begin{align*}
\abs{\ip{\D I^k \left( v \right), u}_{\cH_0^1}} 
    & = \abs{\ip{\D I^k \left( v \right), u}_{\cH_0^1} - \ip{\D I^k \left( 0 \right), u}_{\cH_0^1} + \ip{\D I^k \left( 0 \right), u}_{\cH_0^1}} \\ 
    & \leq \abs{\ip{\D I^k \left( v \right), u}_{\cH_0^1} - \ip{\D I^k \left( 0 \right), u}_{\cH_0^1}} + \abs{\ip{\D I^k \left( 0 \right), u}_{\cH_0^1}} \\
    & \leq K \left( 1 + \norm{v}_{\cH_0^1 } \right)  \norm{u}_{\cH_0^1 } . \qedhere
\end{align*}
\end{proof}

\begin{lemma}
\label{lem:EVtN}
Assume that the neural network satisfies \cref{ass:initial} and the operators of the PDE \eqref{eq:general_operator}--\eqref{eq:operator} satisfy \cref{ass:bound,ass:ellip}.
Then, we have the following inequalities, for all $t\ge 0,$
\begin{align*}
\E \left[ \norm{V^n_t}_{\cH_0^1 }^2 \right] \le C_\psi 
    \quad \text{ and } \quad 
\norm{V_t}^2_{\cH_0^1} \le C_\psi ,
\end{align*}
where $C_\psi$ is a positive constant that only depends on the activation function $\psi$.
\end{lemma} 

\begin{proof}
Let us first show that $I^k \left( V_t^n \right)$ is not increasing in $t$. 
According to \eqref{eq:dytheta}, we have
\[
    \frac{\ud I^k \left( V_t^n \right)}{\ud t} 
        = \nabla_\theta I^k \left( V_t^n \right) \cdot\frac{\ud \theta_t^n}{\ud t} 
        = - \eta_n\abs{\nabla_\theta I^k \left( V^n \left( \theta_t^n; x \right) \right)}^2
        \le 0.
\]
This inequality readily implies $ I^k \left( V_t^n \right)\le I^k \left( V_0^n \right)$. 
Using \cref{ass:bound}, the Cauchy--Schwarz inequality and $ab \leq \frac{a^2 + b^2}{2}$, we get that
\begin{align*}
I^k (u) 
    &= \frac{1}{2}\norm{ u - U^{k - 1}}_{L^2}^2 + \frac{h}{2}\ip{\cL u,u}_{\cH^{-1},\cH_0^1} + h\ip{F \left( U^{k - 1} \right), u}_{L^2}  \\ 
    &= \frac{1}{2}\norm{u}^2_{L^2}+ \frac{h}{2}\ip{\cL u,u}_{\cH^{-1},\cH_0^1} + \ip{hF \left( U^{k - 1}\right) -U^{k-1}, u}_{L^2}+\frac{1}{2}\norm{U^{k-1}}_{L^2}^2\\ 
    &\leq \frac{1}{2}\norm{u}^2_{L^2}+ \frac{hM}{2}\norm{u}^2_{\cH_0^1}+\norm{hF(U^{k-1})-U^{k-1}}_{L^2} \norm{u}_{L^2}+\frac{1}{2} \norm{U^{k-1}}_{L^2}^2 \\ 
    &\leq   \norm{u}^2_{L^2}+ \frac{hM}{2}\norm{u}^2_{\cH_0^1} + \frac{1}{2} \norm{hF(U^{k-1})-U^{k-1}}_{L^2}^2+\frac{1}{2}\norm{U^{k-1}}_{L^2}^2\\
    &= C_1 \norm{u}_{\cH_0^1 }^2 + C_2.
 \end{align*}
Moreover, using \cref{ass:ellip}, the Cauchy--Schwarz inequality again and the inequality 
\[
    \norm{m}_{L^2} \norm{n}_{L^2} \leq \frac{\lambda}{2} \norm{m}_{L^2}^2 + \frac{1}{2 \lambda} \norm{v}_{L^2}^2 
        \text{ with } \lambda = \frac{h \lambda_1}{2},
\]
we arrive at
\begin{align*}
I^k (u) 
    &= \frac{1}{2}\norm{u}^2_{L^2}+ \frac{h}{2}\ip{\cL u,u}_{\cH^{-1},\cH_0^1} + \ip{hF \left( U^{k - 1}\right) -U^{k-1}, u}_{L^2}+\frac{1}{2}\norm{U^{k-1}}_{L^2}^2\\ 
    &\geq \left( \frac{1}{2} - \frac{h \lambda_2}{2} \right) \norm{u}^2_{L^2} + \frac{h \lambda_1}{2} \norm{ u}_{\cH_0^1}^2 - \norm{h F \left( U^{k - 1} \right) - U^{k - 1}}_{L^2} \norm{u}_{L^2} -\frac{1}{2} \norm{U^{k - 1}}_{L_2}^2 \\
    &\geq \left( \frac{1}{2} - \frac{h \lambda_2}{2} \right) \norm{u}^2_{L^2} + \frac{h \lambda_1}{4} \norm{ u}_{\cH_0^1}^2 - \frac{1}{h\lambda_1} \norm{h F \left( U^{k - 1} \right) - U^{k - 1}}_{L^2}^2  -\frac{1}{2} \norm{U^{k - 1}}_{L^2}^2\\
    &= C_3 \norm{u}_{\cH_0^1 }^2 - C_4.
 \end{align*}
Therefore, since $I^k \left( V_t^n \right)$ is not increasing with $t$ and using \cref{lem:v0}, we have that
\begin{align*}
\E \left[ \norm{V^n_t}_{\cH_0^1}^2 \right] 
    &\leq \E \left[ \frac{1}{C_3} I^k \left( V^n_t \right) + \frac{C_4}{C_3} \right] 
     \leq \E \left[ \frac{1}{C_3} I^k \left( V_0^n \right) + \frac{C_4}{C_3} \right] \\
    &\leq \E \left[ \frac{C_1}{C_3} \norm{V^n_0}_{\cH_0^1 }^2 + \frac{C_2 + C_4}{C_3} \right] 
     \leq C_{\psi}.
\end{align*}
Using similar arguments, we get
\begin{align*}
\frac{\ud I^k \left( V_t \right)}{\ud t} 
    = \ip{\mathcal{D} I^k \left( V_t \right), \frac{\ud V_t}{\ud t}}_{\cH_0^1} 
    = - \abs{\ip{\mathcal{D} I^k \left( V_t \right), \E\left[\nabla _\theta\beta_0 \psi \left( \alpha_0 x + c_0 \right)\right]}_{\cH_0^1}}^2
    \le 0.
\end{align*}
This inequality yields that $ I^k \left( V_t \right)\le I^k \left( V_0 \right)=I^k \left( 0 \right)$, hence $\norm{V_t}^2_{\cH_0^1} \le  C_\psi.$
\end{proof}


\subsection{Gradient \texorpdfstring{$\theta$}{θ} estimates}

This subsection contains several useful results concerning gradient estimates of the neurons of the neural network with respect to its parameters.
Moreover, we show that gradients of the neurons are Lipschitz continuous with respect to the parameters $\theta$ of the network, that they converge to their ``unclipped" analogs for large $n$, and we estimate the distance between the parameters as the training progresses.

\begin{lemma}[$\cH_0^1$-boundedness of $X^n$]
\label{lem:boundXN}
Let $\theta \in \R \times \R \times \R^d$, then $X^n$ is bounded in $\cH_0^1$, \emph{i.e.} there exists a constant $C_\psi>0$ such that
\begin{align*}
\norm{X^n(\theta)}^2_{\cH_0^1} &\le C_\psi \left(r_n\right)^{d+8}.
\end{align*}
\end{lemma}

\begin{proof}
The derivatives of a neuron with respect to its parameters equal
\begin{align*}
\frac{\partial}{\partial \beta} \hat\beta \psi \left( \hat\alpha x + \hat c \right)
    &= \psi \left( \hat\alpha x + \hat c \right) \mathbf{1}_{ \{ \abs{\beta} \leq r_n \} }, \\
\frac{\partial}{\partial \alpha} \hat\beta \psi \left( \hat\alpha x + \hat c \right) 
    &= \hat\beta x^\mathsf{T} \left( \nabla \psi \right) \left( \hat\alpha x + \hat c \right) \mathbf{1}_{\{ \frac{1}{r_n} \leq \abs{\alpha} \leq r_n \} }, \\ 
\frac{\partial}{\partial c} \hat\beta \psi \left( \hat\alpha x + \hat c \right) 
    &= \hat\beta \left( \nabla \psi \right) \left( \hat\alpha x + \hat c \right) \mathbf{1}_{ \{ \abs{c} \leq r_n \}}.
\end{align*}
Therefore, we obtain the bound
\begin{align*}
\abs{X^n(\theta)} 
    &= \abs{\frac{\partial}{\partial \beta} \hat\beta \psi \left( \hat\alpha x + \hat c \right) +  \frac{\partial}{\partial \alpha} \hat\beta \psi \left( \hat\alpha x + \hat c \right) + \frac{\partial}{\partial c} \hat\beta \psi \left( \hat\alpha x + \hat c \right)} \\
    &\leq \abs{\psi \left( \hat\alpha x + \hat c \right) \mathbf{1}_{ \{ \abs{\beta} \leq r_n \} }} + \abs{\hat\beta x^\mathsf{T} \left( \nabla \psi \right) \left( \hat\alpha x + \hat c \right) \mathbf{1}_{\{ \frac{1}{r_n} \leq \abs{\alpha} \leq r_n \} }} + \abs{\hat\beta \left( \nabla \psi \right) \left( \hat\alpha x + \hat c \right) \mathbf{1}_{\{\abs{c} \leq r_n \} }} \\
    &\leq \abs{\psi \left( \hat\alpha x + \hat c \right)} + r_n \abs{ x\cdot\nabla \psi  \left( \hat\alpha x + \hat c \right)} + r_n \abs{\left( \nabla \psi \right) \left( \hat\alpha x + \hat c \right)}.
\end{align*}
The second term above can be bounded by 
\begin{align*}
r_n \abs{ x\cdot\nabla \psi  \left( \hat\alpha x + \hat c \right)}
    \le r_n (\hat\alpha_n)^{-1}\Big(\abs{ (\hat\alpha_n x+\hat c)\cdot\nabla \psi  \left( \hat\alpha x + \hat c \right)} + \abs{\hat c \cdot \nabla \psi  \left( \hat\alpha x + \hat c \right)} \Big)
    \le C_\psi (r_n)^3.
\end{align*}
Therefore, using that $\abs{\hat\alpha} \ge (r_n)^{-1}$, we have
\begin{multline*}
\int_{\R^d} \abs{X^n(\theta)}^2\ud x \\
    \le \int_{\R^d} \abs{\psi \left( \hat\alpha x + \hat c \right)}^2 \ud x + (r_n)^2 \int_{\R^d} \abs{x}^2 \abs{ (\nabla \psi)  \left( \hat\alpha x + \hat c \right)}^2 \ud x + (r_n)^2 \int_{\R^d} \abs{\left( \nabla \psi \right) \left( \hat\alpha x + \hat c \right)}^2 \ud x \\
    \stackrel{\left(y=\hat\alpha x\right)}{=} \abs{\hat\alpha}^{-d}\int_{\R^d} \abs{\psi \left(  y + \hat c \right)}^2 \ud y + (r_n)^2 \abs{\hat\alpha}^{-d-2}\int_{\R^d}\abs{y}^2 \abs{\left( \nabla \psi \right) \left(  y + \hat c \right)}^2\ud y \\
   \quad +(r_n)^2 \abs{\hat\alpha}^{-d}\int_{\R^d} \abs{\left( \nabla \psi \right) \left(  y + \hat c \right)}^2\ud y\\
  \stackrel{\left(z= y+ \hat c\right)} {\le} (r_n)^d\int_{\R^d} \abs{\psi \left(  z \right)}^2 \ud z+(r_n)^{d+4}\int_{\R^d}\abs{z - \hat c}^2 \abs{\left( \nabla \psi \right) \left(  z \right)}^2\ud z+(r_n)^{d+2}\int_{\R^d} \abs{\left( \nabla \psi \right) \left( z \right)}^2\ud z\\
 \le C_\psi (r_n)^{d+6}
\end{multline*}
Analogously we can estimate the gradient, and we get that, 
\begin{align*}
\abs{\nabla_x X^n(\theta)}
    \le 2 r_n\abs{\nabla\psi \left( \hat\alpha x + \hat c \right)} + \left(r_n\right)^2 \abs{x}\abs{\left( D^2 \psi \right) \left( \hat\alpha x + \hat c \right)} + (r_n)^2 \abs{\left( D^2 \psi \right) \left( \hat\alpha x + \hat c \right)},
\end{align*}
hence 
\[
    \int_{\R^d} \abs{\nabla_x X^n(\theta)}^2\ud x\le C_\psi (r_n)^{d+8}.
\]
Finally, we arrive at
\[
    \norm{X^n(\theta)}^2_{\cH_0^1}= \int_{\R^d} \abs{X^n(\theta)}^2+\abs{\nabla_xX^n(\theta)}^2\ud x \le C_\psi \left(r_n\right)^{d + 8}. \qedhere
\]
\end{proof}

\begin{lemma}[$\cH_0^1$-boundedness of $X$]
\label{lem:bound_XH}
Let $\theta\in \R\times\R\times\R^d$, and assume that the neural network satisfies \cref{ass:initial}.
Then, $X$ is bounded in $\cH_0^1$, \emph{i.e.} there exists a constant $C_\psi>0$ such that
\begin{align*}
    \norm{X(\theta)}^2_{\cH_0^1} 
        \le C_\psi\left(1+\beta^2\right) \left(1+c^2\right)\left(\abs{\alpha}^{-d}+\abs{\alpha}^{-d-2}+\abs{\alpha}^{2-d}\right).
\end{align*}
Moreover,
\begin{align*}
    \E \left[ \norm{X}^2_{\cH_0^1} \right] < + \infty 
       \quad \text{ and } \quad
    \sup_{n \ge 1} \E \left[ \norm{X^n}^2_{\cH_0^1} \right] < + \infty.
\end{align*}
\end{lemma}

\begin{proof}
Let us first consider the $L^2$-norm of $X(\theta)$. 
As in the proof of \cref{lem:boundXN}, we have
\begin{align*}
\norm{X(\theta)}_{L^2}^2 
&\le \int_{\R^d} \abs{\psi \left( \alpha x + c \right)}^2 \ud x + \beta^2 \int_{\R^d} \abs{x}^2 \abs{\left( \nabla \psi \right) \left( \alpha x + c \right)}^2 \ud x + \beta^2 \int_{\R^d} \abs{\left( \nabla \psi \right) \left( \alpha x + c \right)}^2 \ud x \\
&= \abs{\alpha}^{-d}\int_{\R^d} \abs{\psi \left(  z \right)}^2 \ud z+\beta^2 \abs{\alpha}^{-d-2}\int_{\R^d}\abs{z-c}^2 \abs{\left( \nabla \psi \right) \left(  z \right)}^2\ud z+\beta^2 \abs{\alpha}^{-d}\int_{\R^d} \abs{\left( \nabla \psi \right) \left( z \right)}^2\ud z \\
&\le C_\psi \left(1+\beta^2\right) \left(1+\abs{c}^2\right) \left(\abs{\alpha}^{-d}+\abs{\alpha}^{-d-2}\right),
\end{align*}
where we have used the change of variables $z= \alpha x+c$ for the equality in the second step. 
Analogously, we have that 
\[
    \norm{\nabla X(\theta)}_{L^2}^2 \le C_\psi \left( 1+\beta^2 \right) \left( 1+\abs{c}^2 \right) \left( \abs{\alpha}^{-d} + \abs{\alpha}^{2-d} \right).
\]
Combining these two results, we recover the $\mathcal H^1_0$-estimate of $X(\theta)$. 
Taking expectations on the $\mathcal H^1_0$-estimate of $X(\theta)$ and using \cref{ass:initial}, we have that
\begin{align*}
\E \left[ \norm{X}^2_{\cH_0^1} \right] 
    &\le C_\psi \E \left[ 1 + \abs{\beta^1_0}^2 \right] \E \left[ 1 + \abs{c^1_0}^2 \right] \E \left[ \abs{\alpha_0^1}^{-d} + \abs{\alpha_0^1}^{-d-2} + \abs{\alpha_0^1}^{2-d} \right] < + \infty, 
\intertext{while using the $\mathcal H^1_0$-estimate of $X^n(\theta)$ from the previous lemma, we arrive at}    
\E \left[ \norm{X^n}^2_{\cH_0^1} \right] 
    &\le C_\psi \E \left[ 1 + \abs{\hat\beta^1_0}^2 \right] \E \left[ 1 + \abs{\hat c^1_0}^2 \right] \E \left[ \abs{\hat\alpha_0^1}^{-d} + \abs{\hat\alpha_0^1}^{-d-2} + \abs{\hat\alpha_0^1}^{2-d} \right] \\
    &\le  C_\psi \underbrace{\E \left[ 1 + \abs{\beta^1_0}^2 \right] \E \left[ 1 + \abs{ c^1_0}^2 \right]}_{<+\infty} \E \left[ \abs{\hat\alpha_0^1}^{-d} + \abs{\hat\alpha_0^1}^{-d-2} + \abs{\hat\alpha_0^1}^{2-d} \right].
\end{align*}
Using a similar reasoning as in the proof of \cref{lem:v0}, if $d=1,2$,
\[
\E \left[ \abs{\hat\alpha_0^1}^{-d} + \abs{\hat\alpha_0^1}^{-d-2} + \abs{\hat\alpha_0^1}^{2-d} \right] \le \E \left[ \abs{\alpha_0^1}^{-d} + \abs{\alpha_0^1}^{-d-2} + \left(r_n\right)^{-d} + \left(r_n\right)^{-d-2} + \abs{\alpha_0^1} + 1 \right] < \infty,
\]
and if $d \ge 3$,
\begin{align*}
\E \left[ \abs{\hat\alpha_0^1}^{-d} + \abs{\hat\alpha_0^1}^{-d-2} + \abs{\hat\alpha_0^1}^{2-d} \right] 
    &\le \E \left[ \abs{\alpha_0^1}^{-d} + \abs{\alpha_0^1}^{-d-2} + \left(r_n\right)^{-d} + \left(r_n\right)^{-d-2} + \abs{\alpha_0^1}^{2-d}+\left(r_n\right)^{2-d} \right] \\
    &< \infty.
\end{align*}
Therefore, $\sup_{n \ge 1} \E \left[ \norm{X^n}^2_{\cH_0^1} \right] < + \infty.$
\end{proof}

\begin{lemma}
\label{lem:epN}
Assume that the neural network satisfies \cref{ass:initial}.
Then
\begin{align*}
    \e_n := \E \left[ \norm{X^n - X}_{\cH_0^1}^2 \right] \xrightarrow[n\to\infty]{} 0.
\end{align*}
\end{lemma}

\begin{proof}
Let us decompose $\e_n$ as follows, 
\begin{align*}
\e_n = \E \left[ \norm{X^n - X}_{\cH_0^1}^2 \left\{ \boldsymbol{1}_{\abs{\beta_0^1} \le r_n, \abs{\alpha_0^1} \le r_n, \abs{c_0^1}\le r_n} 
    + \boldsymbol{1}_{\abs{\beta_0^1} > r_n, \abs{\alpha_0^1} \le r_n,\abs{c_0^1}\le r_n } 
    + \boldsymbol{1}_{\abs{\alpha_0^1} > r_n, \abs{c_0^1}\le r_n}
    + \boldsymbol{1}_{ \abs{c_0^1}> r_n} \right\} \right],
\end{align*}
and then we treat each summand separately.

\noindent\textit{Term 1.} 
By definition $\norm{X^n - X}_{\cH_0^1}^2 \boldsymbol{1}_{\{ \abs{\beta_0^1} \le r_n, \abs{\alpha_0^1} \le r_n , \abs{c_0^1}\le r_n \}} = 0.$

\noindent\textit{Term 2.}
Let $\abs{\beta_0^1} > r_n$ and $\abs{\alpha_0^1} \le r_n$, then $\norm{X^n}_{\cH_0^1} \le \norm{X}_{\cH_0^1}$. 
Hence,
\[
\E \left[ \norm{X^n - X}_{\cH_0^1}^2 \boldsymbol{1}_{\{ \abs{\beta_0^1} > r_n, \abs{\alpha_0^1} \le r_n,\abs{c_0^1}\le r_n  \}} \right]\le 2\E \left[ \norm{X}_{\cH_0^1}^2 \boldsymbol{1}_{\{ \abs{\beta_0^1} > r_n, \abs{\alpha_0^1} \le r_n ,\abs{c_0^1}\le r_n \}} \right],
\]
which converges to $0$ by the dominated convergence theorem.

\noindent\textit{Terms 3 \& 4.}
The following inequality holds in this case
\begin{align*}
\norm{X^n - X}_{\cH_0^1}^2 & \Big(\boldsymbol{1}_{\{ \abs{\alpha_0^1} > r_n,\abs{c_0^1}\le r_n  \}}+\boldsymbol{1}_{\abs{c_0^1}> r_n}\Big) \\
    &\le  2 \left( \norm{X^n}_{\cH_0^1}^2 + \norm{X}_{\cH_0^1}^2 \right) \Big(\boldsymbol{1}_{\{ \abs{\alpha_0^1} > r_n,\abs{c_0^1}\le r_n  \}}+\boldsymbol{1}_{\abs{c_0^1}> r_n}\Big).
\end{align*}
Using the dominated convergence theorem, we have
\begin{align*}
    \E \left[ \norm{X}_{\cH_0^1}^2 \Big(\boldsymbol{1}_{\{ \abs{\alpha_0^1} > r_n,\abs{c_0^1}\le r_n  \}}+\boldsymbol{1}_{ \abs{c_0^1}> r_n\}}\Big) \right] \xrightarrow[n\to\infty]{} 0.
\end{align*}
Moreover, applying \cref{lem:boundXN}, we arrive at
\begin{align*}
    \norm{X^n}_{\cH_0^1}^2 \Big(\boldsymbol{1}_{\{ \abs{\alpha_0^1} > r_n\}}+\boldsymbol{1}_{\{ \abs{c_0^1} > r_n\}}\Big)
     \le  C_\psi (r_n)^{d+8}\Big(\boldsymbol{1}_{\{ \abs{\alpha_0^1} > r_n\}}+\boldsymbol{1}_{\{ \abs{c_0^1} > r_n\}}\Big) ,
\end{align*}
which converges to $0$ almost surely. 
Therefore, using the dominated convergence theorem once again, we get 
\[
\E \left[ \norm{X^n}_{\cH_0^1}^2 \Big(\boldsymbol{1}_{\{ \abs{\alpha_0^1} > r_n\}}+\boldsymbol{1}_{\{ \abs{c_0^1} > r_n\}}\Big) \right]
    \le \E\Big[\abs{\alpha_0^1}^{d+8}\boldsymbol{1}_{\{ \abs{\alpha_0^1} > r_n\}}+\abs{c_0^1}^{d+8}\boldsymbol{1}_{\{ \abs{c_0^1} > r_n\}}\Big] 
    \xrightarrow[n\to\infty]{} 0. \qedhere
\]
\end{proof}

\begin{lemma}[$\theta$-Lipschitz continuity]
\label{lem:lip_Xtheta}
Let $\theta,\theta'\in \R\times\R\times\R^d$, then $X^n$ is Lipschitz continuous in $\cH_0^1$, \emph{i.e.} there exists a constant $C_\psi>0$ such that
\begin{align*}
    \norm{X^n(\theta) - X^n(\theta')}_{\cH_0^1} \le C_\psi \left(r_n\right)^{4+\frac{d}{2}}\abs{\theta-\theta'}^{\frac{1}{2}}.
\end{align*}
\end{lemma}

\begin{proof}
As in the proof of \cref{lem:boundXN}, we have
\begin{align}
\label{eq:appendix-bound}
\norm{X^n(\theta) - X^n\left( \theta'\right)}_{\cH_0^1} 
    &\le \norm{\psi \left( \hat\alpha \cdot + \hat c \right)-\psi \left( \hat\alpha' \cdot + \hat c'\right)}_{\cH_0^1} \nonumber \\ \nonumber
    &\qquad  +\norm{\hat\beta  \left( \cdot\nabla \psi \right) \left( \hat\alpha \cdot + \hat c \right)-\hat\beta' \cdot \left( \nabla \psi \right) \left( \hat\alpha' \cdot + \hat c' \right)}_{\cH_0^1} \\
    &\qquad + \norm{\hat\beta \left( \nabla \psi \right) \left( \hat\alpha \cdot + \hat c \right)-\hat\beta' \left( \nabla \psi \right) \left( \hat\alpha' \cdot + \hat c' \right)}_{\cH_0^1}.
\end{align}
Let us recall that $\abs{\hat \alpha}, |\hat\beta|,\abs{\hat c}\le r_n$ and $\abs{\hat \alpha}\ge r_n^{-1}$. 
Using that $\psi \in C_c^{\infty} \left( \mathbb{R}^d \right)$, and therefore Lipschitz, we get
\begin{align*}
\int_{\R^d} \abs{\psi \left( \hat\alpha x + \hat c \right)-\psi \left( \hat\alpha' x + \hat c'\right)}^2\ud x 
    & \le \int_{\R^d} C_\psi \abs{\left(\hat\alpha - \hat \alpha' \right) x + \hat c - \hat c'} \abs{\psi \left( \hat\alpha x + \hat c \right) - \psi \left( \hat\alpha' x + \hat c'\right)} \ud x \\
    & \le C_\psi \abs{\theta-\theta'} \int_{\R^d} \left( 1 + \abs{x} \right) \left( \abs{\psi \left( \hat\alpha x + \hat c \right)}+ \abs{\psi \left( \hat\alpha' x + \hat c'\right)} \right) \ud x.
\end{align*}
Using the change of variables $y = \hat \alpha x + \hat c$, we arrive at
\begin{align*}
    \int_{\R^d} \left( 1 + \abs{x} \right) \abs{\psi \left( \hat\alpha x + \hat c \right)} \ud x & = \abs{\hat \alpha}^{-d} \int_{\R^d} \abs{\psi (y)} \ud y + \abs{\hat \alpha}^{-d-1} \int_{\R^d} \abs{y - \hat c} \abs{\psi \left( y \right)} dy \\
    & \le C_\psi \Big( r_n^{d} + r_n^{d+2} \Big),
\end{align*}
and we obtain the bound
\[
    \int_{\R^d} \abs{\psi \left( \hat\alpha x + \hat c \right)-\psi \left( \hat\alpha' x + \hat c'\right)}^2\ud x \leq C_\psi \abs{\theta-\theta'} r_n^{d+2}.
\]
Analogously, using that $\nabla \psi$ is Lipschitz as well, we have that
\begin{multline*}
\int_{\R^d} \abs{\hat \alpha\nabla\psi\left( \hat\alpha x + \hat c \right)-\hat \alpha'\nabla\psi \left( \hat\alpha' x + \hat c'\right)}^2\ud x \\
     \le 2\int_{\R^d} \abs{\hat\alpha-\hat \alpha'}^2\abs{\nabla\psi \left( \hat\alpha x + \hat c \right)}^2\ud x + 2\int_{\R^d}\abs{\hat \alpha'}^2\abs{\nabla\psi \left( \hat\alpha x + \hat c \right)-\nabla\psi \left( \hat\alpha' x + \hat c'\right)}^2\ud x \\
      \le 4 r_n \abs{\theta - \theta'} \int_{\R^d} \abs{\nabla\psi \left( \hat\alpha x + \hat c \right)}^2\ud x + 2 r_n^2 \int_{\R^d} \abs{\nabla\psi \left( \hat\alpha x + \hat c \right)-\nabla\psi \left( \hat\alpha' x + \hat c'\right)}^2\ud x \\
     \le  C_\psi \abs{\theta-\theta'}(r_n)^{d+4}.
\end{multline*}
Therefore, we arrive at the following bound for the $\cH_0^1$-norm of this term
\[
    \norm{\psi \left( \hat\alpha \cdot + \hat c \right)-\psi \left( \hat\alpha' \cdot + \hat c'\right)}_{\cH_0^1}\le C_\psi \abs{\theta-\theta'}^{1/2}(r_n)^{d/2+2}.
\]
We can similarly estimate the other two terms in \eqref{eq:appendix-bound}, and we get that
\begin{align*}
&\norm{\hat\beta  \left( \cdot\nabla \psi \right) \left( \hat\alpha \cdot + \hat c \right)-\hat\beta' \cdot \left( \nabla \psi \right) \left( \hat\alpha' \cdot + \hat c' \right)}_{\cH_0^1} \\
    &\qquad + \norm{\hat\beta \left( \nabla \psi \right) \left( \hat\alpha \cdot + \hat c \right)-\hat\beta' \left( \nabla \psi \right) \left( \hat\alpha' \cdot + \hat c' \right)}_{\cH_0^1} 
    \le  C_\psi \abs{\theta-\theta'}^{1/2}(r_n)^{d/2+4}
\end{align*}
Concluding, we have
\[
    \norm{X^n(\theta) - X^n\left( \theta'\right)}_{\cH_0^1} \le C_\psi \left(r_n\right)^{4+\frac{d}{2}}\abs{\theta-\theta'}^{\frac{1}{2}}. \qedhere
\]
\end{proof}

\begin{lemma}
\label{lem:thetat0}
Assume that the neural network satisfies \cref{ass:initial}.
Then, for $t\ge0$, we have
\begin{align*}
    \E\left[ \abs{\theta_t^{i, n} - \theta_0^{i, n}}\right] \le C_\psi tn^{\delta-1}\left(r_n\right)^{\frac{d}{2}+4}.
\end{align*} 
\end{lemma}

\begin{remark}
This lemma yields that, when $n$ is large, then the value $\theta_t^{i,n}$ of the evolution of the parameters of the neural network does not differ significantly from its initial value $\theta_0^{i,n}$.
\end{remark}

\begin{proof}
Recalling \eqref{eq:dytheta}, we have 
\begin{align*}
\theta_t^{i, n} - \theta_0^{i,n} 
    = - \eta_n \int_0^t \ip{\D I^k \left( V_s^n \right), \nabla _{\theta^i} V_s^n}_{\cH_0^1} \ud s,
\end{align*}
thus, their squared difference equals
\begin{align*}
 \abs{\theta_t^{i, n} - \theta_0^{i, n}}^2 = & \abs{\eta_n}^2 \left| \int_0^t \ip{\D I^k \left( V_s^n \right), \nabla _{\theta^i} V_s^n}_{\cH_0^1} \ud s \right|^2 \leq n^{4 \delta - 2} t\int_0^t  \ip{\D I^k \left( V_s^n \right), \nabla _{\theta^i} V_s^n}^2_{\cH_0^1}  \ud s .
\end{align*}
Using \cref{{lem:con_frechet}}, we get
\begin{align*}
\left| \ip{\D I^k \left( V_s^n \right), \nabla _{\theta^i} V_s^n}_{\cH_0^1} \right|^2 
    \leq K \left( 1+\norm{V^n_s}^2_{\cH_0^1} \right) \norm{\nabla _{\theta^i} V_s^n}^2_{\cH_0^1 } .
\end{align*}
Using that
\[
\left| \nabla _{\theta^i} V^n \left( \theta^{n}_t; x \right) \right| = \frac{1}{n^{\delta}} \abs{\mathcal X^n\left(\theta^{i,n}_t;x\right)}, 
\]
and applying \cref{lem:boundXN}, we get
\[
    \norm{\nabla _{\theta^i} V_t^n}^2_{\cH_0^1 } \le \frac{C_\psi} {n^{2 \delta}} \left(r_n\right)^{d+8}.
\]
Moreover, from \cref{lem:EVtN}, we have
\[
    \E \left[ \norm{V^n_t}^2_{\cH_0^1 } \right] \le C_\psi.
\]
Hence, we can bound the norm of the square difference by
\[
    \E\left[ \abs{\theta_t^{i, n} - \theta_0^{i, n}}^2 \right] 
        \le C_\psi t n^{2\left(\delta - 1\right)} \left(r_n\right)^{ d+8} \int_0^t \E \left[ \norm{V^n_s}^2_{\cH_0^1 } +1\right] \ud s 
        \le C_\psi t^2 n^{2\left(\delta-1\right)}\left(r_n\right)^{d + 8}.
\]
The result follows now using Jensen's inequality.
\end{proof}

\subsection{Examples}
\label{appendix:examples}

This final appendix contains certain details in order to verify that the option pricing PDEs of \cref{ex:option,ex:op-PIDE} indeed satisfy \cref{ass:bound,ass:ellip,ass:Lip,ass:SA-PD}.

Starting from the Black--Scholes PDE of \cref{ex:option}, we have that 
\begin{align*}
\mathcal{L}u  = - \frac{\sigma^2}{2} \frac{\partial^2 u}{\partial x^2} + r u 
    \quad \text{ and } \quad
F(u) = \left( \frac{\sigma^2}{2} - r \right) \frac{\partial u}{\partial x},
\end{align*}
and the energy functional takes the form
\[
I^k(u) 
    = \frac{1}{2} \norm{u - U^{k-1}}^2_{L^2} + \frac{h}{2} \int_{\R} \Big\{ \frac{\sigma^2}{2} \left( \frac{\partial u}{\partial x} \right)^2 + r u^2 \Big\}\ud x + h \int_{\R} F \left( U^{k-1} \right) u \ud x.
\]
\cref{ass:SA-PD} is obviously satisfied.
Using the triangle and the Cauchy--Schwarz inequalities, following some straightforward calculations, we arrive at
\begin{align*}
\abs{\ip{\cL  u,  v}_{\cH^{-1},\cH_0^1}} 
     \leq \left(\abs{\frac{\sigma^2}{2}} + \abs{r} \right) \norm{u}_{\cH_0^1} \norm{v}_{\cH_0^1} 
    \quad \text{ and } \quad
\norm{F(u)}_{L^2} 
     \leq \abs{\frac{\sigma^2}{2} - r} \norm{u}_{\cH_0^1}.
\end{align*}
Hence, \cref{ass:bound} is satisfied with $M = \abs{\frac{\sigma^2}{2}} + \abs{r}$.
Moreover, we have that
\[
\ip{\cL  u,  u}_{\cH^{-1},\cH_0^1} 
    \geq \left( \frac{\sigma^2}{2} + r \right) \norm{u}_{\cH_0^1}^2,
\]
therefore \cref{ass:ellip} is satisfied with $\lambda_1 = \frac{\sigma^2}{2} + r > 0$ and $\lambda_2=0$.
\cref{ass:Lip} follows by similar computations.

Turning our attention to the multi-dimensional Merton model of \cref{ex:op-PIDE}, we only need to show that the function $F$, which now contains the integro-differential operator stemming from the jumps of the dynamics, still satisfies \cref{ass:bound}.
Let us start with the integral operator, denoted by $F_\nu$, then we have
\begin{align*}
\norm{F_\nu (u) }_{L^2}^2
    &\le  \int_{\R^d} \abs{ \lambda \int_{\R^d} \left(u\left(x\ue^z\right)-u\left(x\right) \right)\nu\left(\ud z\right)}^2\ud x \\
    &\le 2 \lambda \int_{\R^d}\int_{\R^d} \abs{u\left(x\ue^z\right)}^2\nu\left(\ud z\right)\ud x + 2 \lambda \int_{\R^d} \abs{u\left(x\right)}^2\ud x\\
    &= \frac{2\lambda}{(\sqrt{2\pi})^d}\int_{\R^d}\int_{\R^d} \abs{u\left(x\ue^z\right)}^2\ue^{-\frac{z^2}{2}}\ud z\ud x+2\lambda\norm{u}_{L^2}^2\\
    &= \frac{2\lambda}{(\sqrt{2\pi})^d}\int_{\R^d}\int_{\R^d} \abs{u\left(x\ue^z\right)}^2\ue^{-\frac{z^2}{2}}\ud x\ud z + 2\lambda \norm{u}_{L^2}^2\\
    &= \frac{2\lambda}{(\sqrt{2\pi})^d}\int_{\R^d}\int_{\R^d} \abs{u\left(y\right)}^2\ue^{\frac{-2z-z^2}{2}}\ud y\ud z + 2\lambda\norm{u}_{L^2}^2\\
    &= 2\lambda \left(\ue+1\right)\norm{u}_{L^2}^2
     \le 2\lambda \left(\ue+1\right)\norm{u}_{\cH_0^1 }^2,
\end{align*}
where we have used the properties of the normal distribution, Fubini's theorem for the fourth step, and the change of variables $y=x\ue^z$ for the fifth step.
Therefore,
\begin{align*}
\norm{F(u)}_{L^2}
    \le \abs{\boldsymbol{b}}\norm{\nabla u}_{L^2} +\lambda \norm{F_\nu (u)}_{L^2} 
    \le 2\lambda\left(C+\ue+1\right)\norm{u}_{\cH_0^1 },
\end{align*}
which implies the result.

\bibliographystyle{abbrvnat}
\bibliography{references}


\end{document}